\documentclass [12pt,a4paper,reqno]{amsart} \textwidth 165mm
 \input amssymb.sty
 \usepackage{cases}
 \usepackage{mathtools}

\newtheorem{thm}{Theorem} [section]
\newtheorem*{thm*}{Theorem}

\newtheorem{lem}[thm]{Lemma}

 \newtheorem{defn}[thm]{Definition}

 \theoremstyle{remark}
\newtheorem{remark}[thm]{Remark}
\newtheorem{remark*}[thm]{Remark}
\newtheorem{acknowledgment*}[thm] {Acknowledgment}

\newcommand{\thmref}[1]{Theorem~\ref{#1}}

\newcommand{\lemref}[1]{Lemma~\ref{#1}}

 \renewcommand{\sectionmark}[1]{}

\newcommand{\doe}{\overset{\text{def}}{=}}
 \newcommand{\supp} {\operatorname{supp}}
\newcommand{\loc} {\operatorname{loc}}

 \date{}
\begin{document}

\title{Spaces Admissible for
 the Sturm-Liouville Equation}
\author[N.A. Chernyavskaya]{N.A. Chernyavskaya}\address{Department of Mathematics, Ben-Gurion University of the
Negev, P.O.B. 653, Beer-Sheva, 84105, Israel}
\email{nina@math.bgu.ac.il}
\author[L.A. Shuster]{L.A. Shuster}
 \address{Department of Mathematics,
 Bar-Ilan University, 52900 Ramat Gan, Israel}
 \email{miriam@macs.biu.ac.il}

\begin{abstract}
We consider  the equation
\begin{equation}\label{1}
-y''(x)+q(x)y(x)=f(x),\quad x\in \mathbb R
\end{equation}
where $f\in L_p^{\loc}(\mathbb R),$ $p\in[1,\infty)$ and $0\le q\in L_1^{\loc}(\mathbb R).$ By a solution of \eqref{1} we mean any function $y,$ absolutely continuous together with its derivative and satisfying \eqref{1} almost everywhere in $\mathbb R.$ Let positive and continuous functions $\mu(x)$ and $\theta(x)$ for $x\in\mathbb R$ be given. Let us introduce the spaces
\begin{align*}
L_p(\mathbb R,\mu)&=\left\{f\in L_p^{\loc}(\mathbb R):\|f\|_{L_p(\mathbb R,\mu)}^p=\int_{-\infty}^\infty|\mu(x)f(x)|^pdx<\infty\right\},\\
L_p(\mathbb R,\theta)&=\left\{f\in L_p^{\loc}(\mathbb R):\|f\|_{L_p(\mathbb R,\theta)}^p=\int_{-\infty}^\infty|\theta(x)f(x)|^pdx<\infty\right\}.
\end{align*}

In the present paper, we obtain requirements to the functions $\mu,\theta$ and $q$ under which
\begin{enumerate}
\item[1)] for every function $f\in L_p(\mathbb R,\theta)$ there exists a unique solution \eqref{1} $y\in L_p(\mathbb R,\mu)$ of \eqref{1};
    \item[2)] there is an absolute constant $c(p)\in(0,\infty)$ such that regardless of he choice of a function $f\in L_p(\mathbb R,\theta)$ the solution of \eqref{1} satisfies the inequality
        $$\|y\|_{L_p(\mathbb R,\mu)}\le c(p)\|f\|_{L_p(\mathbb R,\theta)}.$$
        \end{enumerate}
\end{abstract}

\maketitle

\baselineskip 20pt

\section{Introduction}\label{introduction}
\setcounter{equation}{0} \numberwithin{equation}{section}

In the present paper, we consider the equation
\begin{equation}\label{1.1}
-y''(x)+q(x)y(x )=f(x),\quad x\in \mathbb R
\end{equation}
where $f\in L_p^{\loc}(\mathbb R)$, $p\in[1,\infty)$ and
\begin{equation}\label{1.2}
0\le q \in L_1^{\loc}(\mathbb R).
\end{equation}

Our general goal is to determine a space frame within which equation \eqref{1.1}  always has a unique stable solution. To state the problem in  a more precise way, let us fix two positive continuous functions $\mu(x)$ and $\theta(x),$ $x\in\mathbb R,$ a number $p\in[1,\infty)$, and introduce the spaces $L_p(\mathbb R,\mu)$ and $L_p(\mathbb R,\theta):$
\begin{align}
&L_p(\mathbb R,\mu)=\left\{f\in L_p^{\loc}(\mathbb R):\|f\|_{L_p(\mathbb R ,\mu)}^p=\int_{-\infty}^\infty|\mu(x)f(x)|^pdx<\infty\right\}\label{1.3}\\
&L_p(\mathbb R,\theta)=\left\{f\in L_p^{\loc}(\mathbb R):\|f\|_{L_p(\mathbb R ,\theta)}^p=\int_{-\infty}^\infty|\theta(x)f(x)|^pdx<\infty\right\}.\label{1.4}
\end{align}

For brevity, below we write $L_{p,\mu}$ and $L_{p,\theta},$ \ $\|\cdot\|_{p,\mu}$ and $\|\cdot\|_{p,\theta}$, instead of $L_p(\mathbb R,\mu),$ $L_p(\mathbb R,\theta)$ and $\|\cdot\|_{L_p(\mathbb R,\mu)}$, $\|\cdot\|_{L_p(\mathbb R,\theta)},$ respectively (for $\mu=1$ we use the standard notation $L_p$ $(L_p:=L_p(\mathbb R))$ and $\|\cdot\|_p$ $(\|\cdot\|_p:=\|\cdot\|_{L_p}).$
In addition, below by a solution of \eqref{1.1} we understand any function $y,$ absolutely continuous together with its derivative and satisfying equality \eqref{1.1} almost everywhere on $\mathbb R$.

Let us introduce the following main definition (see \cite[Ch.5, \S50-51]{12}:
\begin{defn}\label{defn1.1}
We say that the spaces $L_{p,\mu}$ and $L_{p,\theta}$ make a pair $\{L_{p,\mu},L_{p,\theta}\}$ admissible for equation \eqref{1.1} if the following requirements hold:

I) for every function $f\in L_{p,\theta}$ there exists a unique solution $y\in L_{p,\mu}$ of \eqref{1.1};

II) there is a constant $c(p)\in (0,\infty)$ such that regardless of the choice of a function $f\in L_{p,\theta}$ the solution $y\in L_{p,\mu}$ of \eqref{1.1} satisfies the inequality
\begin{equation}\label{1.5}
\|y\|_{p,\mu}\le c(p)\|f\|_{p,\theta}.
\end{equation}
\end{defn}

Let us in addition we make the following conventions: For brevity we say ``problem I)--II)" or ``question on I)--II)" instead of ``problem (or question) on conditions for the functions $\mu$ and $\theta$ under which requirements I)--II) of Definition \ref{defn1.1} hold." We say ``the pair $\{L_{p,\mu};L_{p,\theta}\}$ admissible for \eqref{1.1}" instead of ``the pair of spaces $\{L_{p,\mu};L_{p,\theta}\}$ admissible for equation \eqref{1.1}", and we often omit the word ``equation" before \eqref{1.1}.
By $c,\, c(\cdot)$ we denote absolute positive constants which are not essential for exposition and may differ even within a single chain of calculations.
Our general requirement \eqref{1.2} is assumed to be satisfied throughout the paper, is not referred to, and does not appear in the statements.

Let us return to Definition \ref{defn1.1}.
The question on the admissibility of the pair $\{L_p,L_p\}$ for \eqref{1.1} was studied in \cite{3,6}  (in \cite{3,6} for $\mu\equiv\theta\equiv1$ in the case where I)--II) were valid, we said that equation \eqref{1.1} is correctly solvable in $L_p.$ We maintain this terminology in the present paper.) Let us quote the main result of \cite{3,6} (in terms of Definition \ref{defn1.1}).

\begin{thm} \label{thm1.2} \cite{3}
The pair $\{L_p,L_p\}$ is admissible for \eqref{1.1} if and only if there is $a\in(0,\infty)$ such that $q_0(a)>0.$ Here
\begin{equation}\label{1.6}
q_0(a)=\inf_{x\in\mathbb R}\int_{x-a}^{x+a}q(t)dt.
\end{equation}
\end{thm}

Below we continue the investigation started in \cite{3,6}.

Our goal is as follows: given equation \eqref{1.1}, to determine requirements to the weights $\mu$ and $\theta$ under which the pair $\{L_{p,\mu};L_{p,\theta}\},$ $p\in[1,\infty),$ is admissible for \eqref{1.1}. Such an approach to the inversion of \eqref{1.1} allows to study this equation also in the case where \thmref{thm1.2} is not applicable, for example, in the following three cases:
\begin{enumerate}
\item[1)]\ $q_0(a)>0$ for some $a\in(0,\infty),$ $f\notin L_p,$ $p\in[1,\infty);$
\item[2)]\ $q_0(a)=0$ for all $a\in(0,\infty),$ $f\in L_p,$ $p\in[1,\infty);$
\item[2)]\ $q_0(a)=0$ for all $a\in(0,\infty),$ $f\notin L_p,$ $p\in[1,\infty).$
\end{enumerate}

 Our main result (see \thmref{thm4.3} in \S4 below) reduces the stated problem to the question on the boundedness of a certain integral operator $S: L_p\to L_p$ (see \eqref{4.3} in \S4). {}From this criterion, under additional requirements to the functions $\mu,$ $\theta$ and $q,$ one can deduce some concrete particular conditions which control the solution of our problem. See \S4 for such restrictions.

 We now describe the structure of the paper. Section~2 contains preliminaries; in Section~3 we give various technical assertions; all our results and relevant comments are presented in Section~4; all the proofs are collected in Section~5; and Section~6 contains an example of the presented statements.

 \section{Preliminaries}

 Recall that our standing assumption \eqref{1.2} is not included in the statements.

 \begin{lem}\label{lem2.1} \cite{4} Suppose that the following condition holds:
 \begin{equation}\label{2.1}
 \int_{-\infty}^x q(t)dt>0,\quad \int_x^\infty q(t)dt>0,\quad \forall x\in\mathbb R.
 \end{equation}
 Then for any given $x\in\mathbb R,$ each of the equations in $d\ge0$
 \begin{equation}\label{2.2}
 \int_0^{\sqrt 2 d}\int_{x-t}^{x+t}q(\xi)d\xi dt=2,\qquad d\int_{x-d}^{x+d}q(\xi)d\xi=2
 \end{equation}
 has a unique finite positive solution. Denote these solutions by $d(x)$ and $\hat d(x)$, respectively. We have the inequalities
 \begin{equation}\label{2.3}
 \frac{d(x)}{\sqrt 2}\le \hat d(x)\le \sqrt 2 d(x),\quad x\in\mathbb R.
 \end{equation}
 \end{lem}

 Note that the functions $d(x)$ and $\hat d(x)$ were introduced by the authors (see \cite{1,4}) and M. Otelbaev (see \cite{10}), respectively. Analysing our assertions and requirements (see \S4 below), it is useful to take into account that the function $q^*(x)\doe d^{-2}(x)$ $(d^{-2}:=1/d^2)$ can be interpreted as a composed (in the sense of function theory) average of the function $q(\xi),$ $\xi\in\mathbb R$, at the point $\xi=x$ with step $d(x).$ Indeed, denote
 $$S_x(q)(t)=\frac{1}{2t}\int_{x-t}^{x+t}q(\xi)d\xi,\quad t>0, \quad x\in\mathbb R,$$
 $$M(f)(\eta)=\frac{1}{\eta^2}\int_0^{\sqrt 2\eta}tf(t)dt,\quad \eta>0.$$
Clearly, $S_x(q)(t)$ is the Steklov average with step $t>0$ of the function $q(\xi),$ $\xi\in\mathbb R$, at the point $\xi=x,$ and $M(f)(\eta)$ is the average of the function $f(t),$ $t>0$ with step $\eta>0$ at the point $t=0.$ Now, using
\begin{align*}q^*(x)&=\frac{1}{d^2(x)}=\frac{1}{2d^2(x)}\int_0^{\sqrt 2d(x)}\int_{x-t}^{x+t}q(\xi)d\xi dt\\
&=\frac{1}{d^2(x)}\int_0^{\sqrt 2 d(x)}t\left[\frac{1}{2t}\int_{x-t}^{x+t}q(\xi)d\xi\right]dt=M(S_x(q))(d(x)).
\end{align*}

Similarly, the function $\hat q^*(x)\doe \hat d(x)^{-2}$ $x\in\mathbb R,$ can be interpreted as the Steklov average of the function $q(\xi),$ $\xi\in\mathbb R,$ at the point $\xi=x$ with step $\hat d(x).$ Indeed (see \eqref{2.1}), we have
$$\hat q^*(x)=\frac{1}{\hat d^2(x)}=\frac{1}{2\hat d(x)}\int_{x-\hat d(x)}^{x+\hat d(x)} q(\xi)d\xi=S_x(q)(\hat d(x)).$$

\begin{thm}\label{thm2.2} \cite{2}
Suppose that \eqref{2.1} holds. Then the equation
\begin{equation}\label{2.4}
z''(x)=q(x)z(x),\quad x\in\mathbb R,
\end{equation}
has a fundamental system of solutions (FSS) $\{u(x),v(x)\},$ $x\in\mathbb R,$ such that
\begin{gather}
u(x)>0,\quad v(x)>0,\quad u'(x)<0,\quad v'(x)>0,\quad \forall x\in\mathbb R,\label{2.5}\\
v'(x)u(x)-u'(x)v(x)=1,\quad \forall x\in\mathbb R,\label{2.6}\\
\lim_{x\to-\infty}\frac{v(x)}{u(x)}=\lim_{x\to\infty}\frac{u(x)}{v(x)}=0,\label{2.7}\\
|\rho'(x)|<1,\quad \forall x\in\mathbb R,\quad \rho(x)\doe u(x)v(x).\label{2.8}
\end{gather}
\end{thm}

Let us introduce the Green function of equation \eqref{1.1}:
\begin{equation}\label{2.9}
G(x,t)=\begin{cases} u(x)v(t),\ & x\ge t\\
u(t)v(x),\ &x\le t\end{cases}
\end{equation}

\begin{thm}\label{thm2.3}\cite{8}
For $x,t\in\mathbb R$, we have the Davies-Harrell representations for the solution $\{u(x),v(x)\}$ and the Green function $G(x,t):$
\begin{gather}\label{2.10}
u(x)=\sqrt{\rho(x)}\exp\left(-\frac{1}{2}\int_{x_0}^x\frac{d\xi}{\rho(\xi)}\right),\quad v(x)=\sqrt{\rho(x)}\exp\left(\frac{1}{2}\int_{x_0}^x\frac{d\xi}{\rho(\xi)}\right),\\
\label{2.11}
G(x,t)=\sqrt{\rho(x)\rho(t)}\exp\left(-\frac{1}{2}\left|\int_x^t\frac{d\xi}{\rho(\xi)}\right|\right).
\end{gather}
\end{thm}

Here $x_0$ is a unique solution of the equation $u(x)=v(x),$ $x\in\mathbb R$ (see \cite{2}), the function $\rho$ is defined in \eqref{2.8}.

\begin{thm}\label{thm2.4} \cite{4}
Suppose that \eqref{2.1} holds. Then we have the Otelbaev ienqualities:
\begin{equation}\label{2.12}
\frac{d(x)}{2\sqrt 2}\le\rho(x)\le\sqrt 2 d(x),\quad x\in\mathbb R.
\end{equation}
\end{thm}

Two-sided, sharp by order estimates of the function $\rho$ were first obtained by M. Otelbaev (see \cite{10}), and therefore all such inequalities are referred to by his name. Note that the inequalities given in \cite{10} are expressed in terms of another auxiliary function, more complicated than $d(x),$ $x\in\mathbb R$, and are proven under auxiliary requirements to the function $q.$

Let us introduce the Green operator
\begin{equation}\label{2.13}
(Gf)(x)=\int_{-\infty}^\infty G(x,t)f(t)dt,\quad x\in\mathbb R.
\end{equation}

\begin{thm}\label{thm2.5} \cite{6}
Suppose that \eqref{2.1} holds, and let $p\in[1,\infty).$  Then  equation \eqref{1.1} is correctly solvable in $L_p$ (or, in other words, the pair $\{L_p,L_p\}$ is admissible for \eqref{1.1}) if and only if the operator $G:L_p\to L_p$ is bounded. In the latter case, for $f\in L_p,$ the solution $y\in L_p$ of \eqref{1.1} is of the form $y=Gf.$
\end{thm}

\begin{thm}\label{thm2.6} \cite{3}
For $p\in[1,\infty)$, equation \eqref{1.1} is correctly solvable in $L_p$ (i.e., the pair
$\{L_p,L_p\}$ is admissible for \eqref{1.1}) if and only if equalities \eqref{2.1} hold and $\hat d_0<\infty.$
Here
\begin{equation}\label{2.14}
\hat d_0=\sup_{x\in\mathbb R}\hat d(x).
\end{equation}
\end{thm}

\begin{thm}\label{thm2.7} \cite{11}
 Let $\mu$ and $\theta$ be continuous positive functions in $\mathbb R,$ and let $H$ be an integral operator
 \begin{equation}\label{2.15}
 (Hf)(t)=\mu(t)\int_t^\infty\theta(\xi)f(\xi)d\xi,\quad t\in\mathbb R.
 \end{equation}
 For $p\in(1,\infty)$, the operator $H:L_p\to L_p$ is bounded if and only if $H_p<\infty.$ Here $H_p=\sup\limits_{x\in\mathbb R}H_p(x),$
  \begin{equation}\label{2.16}
 H_p(x)=\left(\int_{-\infty}^x\mu(t)^pdt\right)^{1/p}\cdot\left(\int_x^\infty\theta(t)^{p'}dt\right)
 ^{1/p'},\quad p'=\frac{p}{p-1}.
 \end{equation}
 In addition,
  \begin{equation}\label{2.17}
 H_p\le\|H\|_{p\to p}\le (p)^{1/p}(p')^{1/p'}H_p.
 \end{equation}
 \end{thm}

 \begin{thm}\label{thm2.8} \cite{11}
  Let $\mu$ and $\theta$ be continuous positive functions in $\mathbb R,$ and let $\tilde H$ be an integral operator
   \begin{equation}\label{2.18}
  (\tilde Hf)(t)=\mu(t)\int_{-\infty}^t\theta(\xi)f(\xi)d\xi,\quad t\in\mathbb R.
 \end{equation}
 For $p\in(1,\infty)$ the operator $\tilde H:L_p\to L_p$ is bounded if and only if $\tilde H_p<\infty.$ Here $\tilde H_p=\sup_{x\in\mathbb R}\tilde H_p(x)$
  \begin{equation}\label{2.19}
  \tilde H_p(x)=\left[\int_{-\infty}^x\theta(t)^{p'}dt\right]^{1/p'}\cdot\left[\int_x^\infty\mu(t)^p dt\right]^{1/p},\quad p'=\frac{p}{p-1}.
 \end{equation}
 In addition,
  \begin{equation}\label{2.20}
 \tilde H_p\le\|\tilde H\|_{p\to p}\le(p)^{1/p}(p')^{1/p'}\tilde H_p.
 \end{equation}
\end{thm}

\begin{thm}\label{thm2.9} \cite{10a}
 Let  $-\infty\le a<b\le \infty,$ let $K(x,t)$ be a continuous function for $s,t\in(a,b), $ and let $K$ be an integral operator
 \begin{equation}\label{2.21}
 (Kf)(t)=\int_a^bK(s,t)f(s)ds,\quad t\in(a,b).
 \end{equation}
 Then we have the inequality
  \begin{equation}\label{2.22}
 \|K\|_{L_1(a,b)\to L_1(a,b)}=\sup_{s\in(a,b)}\int_a^b|K(s,t)|dt.
 \end{equation}
 \end{thm}

 \section{Auxiliary assertions}

 In this section, we mainly present the properties of the function $d(x)$, $x\in\mathbb R$ (see \lemref{lem2.1}). Here we assume that condition \eqref{2.1} is satisfied, and we do not include it in the statements.

 \begin{lem}\label{lem3.1}
 The function $d(x)$ is continuously differentiable for all $x\in\mathbb R,$ and the following inequality holds:
 \begin{equation}\label{3.1}
 \sqrt 2|d'(x)|\le 1,\quad x\in\mathbb R.
 \end{equation}
 \end{lem}

 \begin{remark}\label{rem3.2}
 It is interesting to compare estimate \eqref{2.8} (see also \eqref{2.12}) with estimate \eqref{3.1}.
 \end{remark}

 \begin{lem}\label{lem3.3}
 For $x\in\mathbb R,$ we have the inequalities
 \begin{equation}\label{3.2}
 4^{-1}d(x)\le d(t)\le 4d(x),\quad\text{if}\quad |t-x|\le d(x).
 \end{equation}
 \end{lem}

  \begin{lem}\label{lem3.4}
 For $x\in\mathbb R,$ we have the inequalities (see \thmref{thm2.2}):
 \begin{equation}\label{3.3}
 c^{-1}\le\frac{u(t)}{u(x)};\ \frac{v(t)}{v(x)};\quad \frac{\rho(t)}{\rho(x)}\le c\quad\text{if}\quad |t-x|\le d(x).
 \end{equation}
 \end{lem}

  \begin{lem}\label{lem3.5}
 For a given $x\in\mathbb R,$ consider the function
 \begin{equation}\label{3.4}
 F(\eta)=\int_0^{\sqrt 2\eta}\int_{x-t}^{x+t}q(\xi)d\xi dt,\quad \eta\ge 0.
 \end{equation}
 The function $F(\eta)$ is differentiable and non-negative, together with its derivative, and
 \begin{equation}\label{3.5}
 F(0)=0,\qquad F(\infty)=\infty.
 \end{equation}
 In addition, the inequality $\eta\ge d(x)$ $(0\le\eta\le d(x))$ holds if and only if $F(\eta)\ge2$ $(F(\eta)\le 2).$
 \end{lem}

 \begin{lem}\label{lem3.6}
Let a function $f$ be defined on $\mathbb R$ and absolutely continuous together with its derivative. Then for all $x\in\mathbb R$ and $t\ge0,$ we have the equality
 \begin{equation}\label{3.6}
\int_{x-t}^{x+t}f(\xi)d\xi=2f(x)t+\int_0^t\int_0^{t_1}\int_{x-t_2}^{x+t_2}f''(t_3)dt_3dt_2dt_1.
 \end{equation}
 \end{lem}

 \begin{thm}\label{thm3.7}
 Suppose that condition \eqref{2.1} holds and the function $q(x)$ can be written in the form
 \begin{equation}\label{3.7}
 q(x)=q_1(x)+q_2(x),\quad x\in\mathbb R,
 \end{equation}
 where $q_1(x),$ $x\in\mathbb R,$ is positive and absolutely continuous together with its derivative, and $q_2\in L_1^{\loc}(\mathbb R).$ Denote
  \begin{gather}\label{3.8}
 A(x)=\left[0,\frac{2}{\sqrt{q_1(x)}}\right],\quad x\in\mathbb R,\\
 \label{3.9}
 \varkappa_1(x)=\frac{1}{q_1(x)^{3/2}}\sup_{t\in A(x)}\left|\int_{x-t}^{x+t}q_1''(\xi)d\xi\right|,\quad x\in\mathbb R,\\
 \label{3.10}
 \varkappa_2(x)=\frac{1}{\sqrt{q_1(x)}}\sup_{t\in A(x)}\left|\int_{x-t}^{x+t}q_2(\xi)d\xi\right|,\quad x\in\mathbb R.
 \end{gather}
 If we have the condition
 \begin{equation}\label{3.11}
 \varkappa_1(x)\to0,\quad \varkappa_2(x)\to0\quad\text{as}\quad |x|\to\infty,
 \end{equation}
 then the following relations hold:
 \begin{gather}
 d(x)\sqrt{q_1(x)}=1+\varepsilon(x),\qquad |\varepsilon(x)|\le 2(\varkappa_1(x)+\varkappa_2(x)),\qquad |x|\gg 1,\label{3.12}\\
 c^{-1}\le d(x)\sqrt{q_1(x)}\le c\quad\text{for all}\quad x\in\mathbb R.\label{3.13}
 \end{gather}
 \end{thm}

\section{Main results}

Throughout the sequel we assume that our standing requirements to the functions $q$ (see \eqref{1.2}), and $\mu$ and $\theta$ (see \S1) are satisfied, and we do not mention them in the statements.

\begin{thm}\label{thm4.1}
Suppose that the function $q$ is nonnegative and continuous at every point of the real axis. Suppose that for a given $p\in[1,\infty)$ the following condition holds:
\begin{equation}\label{4.1}
\int_{-\infty}^0\mu(t)^pdt=\int_0^\infty \mu(t)^pdt=\infty.
\end{equation}
Then the pair $\{L_{p,\mu};L_{p,\theta}\}$ is admissible for \eqref{1.1} only if inequalities \eqref{2.1} hold.
\end{thm}

To make our a priori requirements independent of the parameter $p\in[1,\infty)$, throughout the sequel we assume that together with \eqref{1.2}, condition \eqref{2.1} holds. Similar to \eqref{1.2}, below this condition is not quoted and does not appear in the statements.

\begin{lem}\label{lem4.2}
Suppose that the following condition holds:
\begin{equation}\label{4.2}
\int_{-\infty}^0\mu(t)dt=\int_0^\infty \mu(t)dt=\infty.
\end{equation}
Then for every $p\in[1,\infty)$ equation \eqref{2.4} has no solutions $z\in L_{p,\mu}$ apart from $z\equiv0.$
\end{lem}
Note that for $\mu\equiv1$ \lemref{lem4.2} was proved in \cite{2}.

Our main result is the following.

\begin{thm}\label{thm4.3}
Suppose that condition \eqref{4.2} holds. Then the pair $\{L_{p,\mu};L_{p,\theta}\}$ is admissible for \eqref{1.1} if and only if the operator $S:L_p\to L_p$ is bounded. Here
\begin{equation}\label{4.3}
(Sf)(x)=\mu(x)\int_{-\infty}^\infty\frac{G(x,t)}{\theta(t)}f(t)dt,\quad x\in\mathbb R,\quad f\in L_p.
\end{equation}
\end{thm}

Note that for $\mu\equiv\theta\equiv1$ \thmref{thm4.3} was proved in \cite{6}. Thus, this theorem reduces the original problem on the admissibility of the  pair $\{L_{p,\mu};L_{p,\theta}\}$ for \eqref{1.1} to the boundedness of the integral operator $S:L_p\to L_p$ (see \eqref{4.3}). This result is clearly useful for the investigation of \eqref{1.1} for the following reason. Consider, say, the case $p\in(1,\infty)$. The operator $S$ is a sum of two operators of Hardy type (see \eqref{2.9}, \eqref{2.15} and \eqref{2.18}):
\begin{align}
&(S_1f)(x)=\mu(x)u(x)\int_{-\infty}^x\frac{v(t)}{\theta(t)}f(t)dt,\quad x\in\mathbb R,\label{4.4}\\
&(S_2f)(x)=\mu(x)v(x)\int^{\infty}_x\frac{u(t)}{\theta(t)}f(t)dt,\quad x\in\mathbb R.\label{4.5}
\end{align}
For the norms $\|S_1\|_{p\to p},$ $\|S_2\|_{p\to p},$ we know sharp by order two-sided estimates (see \eqref{2.17} and \eqref{2.20}), which can be expressed in terms of the weights $\mu,\theta$ and a FSS $\{u,v\}$ of equation \eqref{2.4}. The solutions $\{u,v\}$ can, in turn, be expressed in terms of the implicit function $\rho$ (see \eqref{2.10}), for which in turn one has sharp by order estimates in terms of the function $d$ (see \eqref{2.12} and \eqref{2.2}). Finally, for the implicit function $d,$ which is, in general, not computable, as well as the function $\rho,$ we have sharp by order two-sided estimates, which can be expressed in terms of the original function $q$ (see \eqref{3.12}, \eqref{3.13}). Thus, this long chain of estimates yields some information allowing us to find conditions for the boundedness of the operator $S_i: L_p\to L_p,$ $i=1,2$ (and hence of the operator $S:L_p\to L_p,$ $p\in (1,\infty)),$ which are expressed in terms of the weights $\mu,\theta$ and the function $q.$ We want to emphasize that these conditions become precise if we are able to use the information obtained from the estimates in an ingenious way (see, say, \cite{6} where similar arguments were used). One can compare this approach to that of applying the Cauchy criterion for the convergence of a number series to getting various working criteria, convenient for practical investigation of a given number series. In a similar way, \thmref{thm4.3} can be used for deducing convenient particular tests for the admissibility of the pair  $\{L_{p,\mu};L_{p,\theta}\}$, $p\in[1,\infty)$, for a given equation \eqref{1.1}.

Here is an example.  The assertion given below (\thmref{thm4.7}) is obtained by using one of the possible ways for practical implementation of the approach to the study of \eqref{1.1} presented above.

To formulate \thmref{thm4.7}, we need some new definitions, auxiliary assertions and comments.

\begin{defn}\label{defn4.4}
We say that the function $q$ belongs to the class $H$ (and write $q\in H)$ if the following equality holds:
\begin{equation}\label{4.6}
\lim_{|x|\to\infty}\nu(x)=0.
\end{equation}
Here
\begin{equation}\label{4.7}
\nu(x)=d(x)\int_0^{\sqrt 2d(x)}(q(x+t)-q(x-t))dt,\quad x\in\mathbb R.
\end{equation}
\end{defn}

In the next assertion, we state an important property of the functions $q\in H.$

\begin{lem}\label{lem4.5}
Let $q\in H.$ Then for any $\varepsilon>0$ there is a constant $c(\varepsilon)\in[1,\infty)$ such that for all $x,t\in\mathbb R$ the following inequalities hold:
\begin{equation}\label{4.8}
c(\varepsilon)^{-1}\exp\left(-\varepsilon\left|\int_x^t\frac{d\xi}{d(\xi)}\right|\right)\le\frac{d(t)}
{d(x)}\le c(\varepsilon)\exp\left(\varepsilon\left|\int_x^t\frac{d\xi}{d(\xi)}\right|\right).
\end{equation}
\end{lem}

Note that for $\varepsilon\ge 1/\sqrt 2$ inequalities \eqref{4.8} hold regardless of condition \eqref{4.6}. Indeed, under conditions \eqref{1.2} and \eqref{2.1}, the function $d(x), $ $x\in\mathbb R$ is well-defined, differentiable, and satisfies the following relations (see Lemmas \ref{lem2.1} and \ref{lem3.1}:
\begin{align}
-\varepsilon\le-\frac{1}{\sqrt2}&\le d'(\xi)\le\frac{1}{\sqrt2}\le\varepsilon,\quad\xi\in\mathbb R\quad\Rightarrow\nonumber\\
-\frac{\varepsilon}{d(\xi)}&\le\frac{d'(\xi)}{d(\xi)} \le\frac{\varepsilon}{d(\xi)},\quad\xi\in\mathbb R \quad\Rightarrow\nonumber\\
\exp\left(-\varepsilon\left|\int_x^t\frac{d\xi}{d(\xi)}\right|\right)&\le\frac{d(t)}{d(x)}\le \exp\left(\varepsilon\left|\int_x^t\frac{d\xi}{d(\xi)}\right|\right),\quad x,t\in\mathbb R.
\label{4.9}
\end{align}
This means that in contrast with \eqref{4.9}, for $\varepsilon\in(0,1\sqrt 2)$ estimates \eqref{4.8} arise because of condition \eqref{4.6}.

\begin{defn}\label{defn4.6}
Let $q\in H.$ We say that a pair of weights (weight functions) $\{\mu,\theta\}$ agrees with the function $q$ if for any $\varepsilon>0$ there is a constant $c(\varepsilon)\in[1,\infty)$ such that for all $t,x\in\mathbb R$ one has the inequalities
\begin{equation}\label{4.10}
c(\varepsilon)^{-1}\exp\left(-\varepsilon\left|\int_x^t\frac{d\xi}{d(\xi)}\right|\right)\le\sqrt{
\frac{d(t)}{d(x)}}\frac{\mu(t)}{\mu(x)};\ \sqrt{\frac{d(t)}{d(x)}}\frac{\theta(x)}{\theta(t)}\le c(\varepsilon)\exp\left(\varepsilon\left|\int_x^t\frac{d\xi}{d(\xi)}\right|\right).
\end{equation}
In the latter case, we say that the pair $\{L_{p,\mu};L_{p,\theta}\},$ $p\in [1,\infty)$, agrees with equation \eqref{1.1}
\end{defn}

\begin{thm}\label{thm4.7}
Suppose that conditions \eqref{4.2} hold. Let $q\in H.$ Suppose that the pair $\{L_{p,\mu};L_{p,\theta}\},$ $p\in [1,\infty)$ agrees with equation \eqref{1.1}. Then this pair is admissible for \eqref{1.1} if and only if $m(q,\mu,\theta)<\infty.$ Here
\begin{equation}\label{4.11}
m(q,\mu,\theta)=\sup_{x\in\mathbb R}\left(\frac{\mu(x)}{\theta(x)}d^2(x)\right).
\end{equation}
\end{thm}

To prove inequalities \eqref{4.10}, the following lemma can be useful.

\begin{lem}\label{lem4.8}
Suppose that a function $\mu(x)$ is defined, positive and differentiable for all $x\in\mathbb R,$ let $q\in H,$ and let $d(x),$ $x\in\mathbb R,$ denote the auxiliary function from \lemref{lem2.1}. Then, if the equality
\begin{equation}\label{4.12}
\lim_{|x|\to\infty}\frac{\mu'(x)}{\mu(x)}d(x)=0
\end{equation}
holds, then for any given $\varepsilon>0$ there is a constant $c(\varepsilon)\in(0,\infty)$ such that for all $t,x\in\mathbb R$ inequalities \eqref{4.10} hold.
\end{lem}

The next assertions are convenient for the study of concrete equations. They are obvious and are given without proofs.

\begin{thm}\label{4.9}
Let $q\in H,$ and suppose that
\begin{gather}
d_0\doe\sup_{x\in\mathbb R}d(x)=\infty,\label{4.13}\\
\int_{-\infty}^0q^*(x)dx=\int_0^\infty q^*(x)dx=\infty,\quad q^*(x)=\frac{1}{d^2(x)},\quad x\in\mathbb R.\label{4.14}
\end{gather}
Then the following assertions hold:
\begin{enumerate}
\item[A)] for $p\in[1,\infty)$ the pair $\{L_p;L_p\}$
is not admissible for \eqref{1.1};
\item[B)] for $p\in[1,\infty)$ the pair $\{L_{p,q^*};L_p\}$ is admissible for \eqref{1.1}.
\end{enumerate}
\end{thm}

\begin{thm}\label{4.10}
Let $q\in H,$ and suppose that the weight function $\theta(x),$ $x\in\mathbb R,$ is such that $m_0>0$ where
\begin{equation}\label{4.15}
 m_0=\inf_{x\in\mathbb R}(q^*(x)\theta(x)),\quad q^*(x)=\frac{1}{d^2(x)}.
\end{equation}
Then  for $p\in[1,\infty)$ the pair $\{dL_p;L_{p,\theta}\}$
is  admissible for \eqref{1.1}.
\end{thm}

\section{Proofs}

\begin{proof}[Proof of \lemref{lem3.1}]

The existence of the derivative $d'(x),$ $x\in\mathbb R$ is a consequence of the theory of implicit functions \cite[Ch.II,\S1,no.3]{7}. It is proven in the same way as in \cite{5}. The following relations are deduced from \eqref{2.2}:
$$\int_0^{\sqrt 2d(x)}\int_{x-t}^{x+t}q(\xi)d\xi dt=2\quad \Rightarrow$$
\begin{align*}
0&=\sqrt2d'(x)\int_{x-\sqrt 2d(x)}^{x+\sqrt2d(x)}q(\xi)d\xi+\int_0^{\sqrt 2 d(x)}[q(x+t)-q(x-t)]dt\\
&=\sqrt2d'(x)\int_{x-\sqrt2d(x)}^{x+\sqrt2d(x)}q(\xi)d\xi+\left[\int_x^{x+\sqrt2d(x)}q(\xi)d\xi-
\int_{x-\sqrt2d(x)}^xq(\xi)d\xi\right]\quad\Rightarrow\\
|d'(x)|&=\frac{1}{\sqrt 2}\left|\int_x^{x+\sqrt2d(x)}q(\xi)d\xi-\int_{x-\sqrt2d(x)}^x q(\xi)d\xi\right|
\left(\int_{x-\sqrt2d(x)}
^{x+\sqrt2d(x)}q(\xi)d\xi\right)^{-1}\le\frac{1}{\sqrt2}.
\end{align*}
\end{proof}

\begin{proof}[Proof of \lemref{lem3.3}]

Below we use Lagrange's formula and \eqref{3.1}:
\begin{align*}
&|d(t)-d(x)|=|d'(\theta)|\, |t-x|\le\frac{d(x)}{\sqrt2}\quad\Rightarrow\\
&d(t)\le\left(1+\frac{1}{\sqrt2}\right)d(x)\le 4d(x)\quad\text{for}\quad t\in[x-d(x),\, x+d(x)]\\
&d(t)\ge\left(1-\frac{1}{\sqrt2}\right)d(x)\ge \frac{d(x)}{4}\quad\text{for}\quad t\in[x-d(x),\, x+d(x)].
\end{align*}
\end{proof}

\begin{proof}[Proof of \lemref{lem3.4}]

Below we use \eqref{2.12} and \eqref{3.2}:
\begin{align*}
 \int_{x-d(x)}^{x+d(x)}\frac{d(\xi)}{\rho(\xi)}&=\int_{x-d(x)}^{x+d(x)}\frac{d(\xi)}{\rho(\xi)}\cdot
\frac{d(x)}{d(\xi)}\cdot\frac{d\xi}{d(x)}\le 2\sqrt 2\cdot 4\cdot 2=c<\infty,\\
 \int_{x-d(x)}^{x+d(x)}\frac{d\xi}{\rho(\xi)}&=\int_{x-d(x)}^{x+d(x)} \frac{d(\xi)}{\rho(\xi)}
  \cdot
\frac{d(x)}{d(\xi)}\cdot\frac{d\xi}{d(x)}\ge \frac{1}{\sqrt 2}\cdot \frac{1}{4}\cdot 2\ge c^{-1}>0.
\end{align*}

Now we use this together with \eqref{2.10} and obtain
\begin{align*}
\frac{u(t)}{u(x)}&\ge\sqrt{\frac{\rho(t)}{\rho(x)}}\exp\left(-\frac{1}{2}\left|\int_x^t\frac{d\xi}
{\rho(\xi)}\right|\right)\ge\sqrt{\frac{d(x)}{\rho(x)}\cdot\frac{d(t)}{d(x)}\cdot\frac{\rho(t)} {d(t)}}\exp\left(-\frac{1}{2}\left|\int_{x-d}^{x+d}\frac{d\xi}{\rho(\xi)}\right|\right)\\
&\ge c^{-1}>0;\\
\frac{u(t)}{u(x)}&\le\sqrt{\frac{\rho(t)}{\rho(x)}}\exp\left(\frac{1}{2}\left|\int_x^t\frac{d\xi}
{\rho(\xi)}\right|\right)\le\sqrt{\frac{d(x)}{\rho(x)}\cdot\frac{d(t)}{d(x)}\cdot\frac{\rho(t)} {d(t)}}\exp\left(\frac{1}{2}\left|\int_{x-d}^{x+d}\frac{d\xi}{\rho(\xi)}\right|\right)\\
&\ge c<\infty.
\end{align*}

Inequalities \eqref{3.3} for the solution $v$ are checked similarly, and estimates \eqref{3.3} for $\rho$ follow from the estimates of $u$ and $v$ and \eqref{2.8}.
\end{proof}

\begin{proof}[Proof of \lemref{lem3.5}]

To prove that the function $F(\eta)$ is differentiable and the functions $F(\eta)$ and $F'(\eta)$ are non-negative for $\eta\ge0$, we use properties of integral. The last assertion of the lemma follows from Lagrange's formula and the relations
$$F(\eta)-2=F(\eta)-F(d(x))=F'(\theta)(\eta-d(x)).$$
\end{proof}

\begin{proof}[Proof of \lemref{lem3.6}]

To obtain \eqref{3.6}, we use the following simple transformations
 \begin{align*}
 \int_{x-t}^{x+t}f(\xi)d\xi&=\int_0^t[f(x+t_1)+f(x-t_1)]dt_1=2f(x)t+\int_0^t[f(x+t_1)-f(x)]dt\\
 &\quad -\int_0^t[f(x)-f(x-t_1)]dt_1=2f(x)t+\int_0^t\left[\int_0^{t_1}(f(x+t_2))'dt_2\right]dt_1\\
 &\quad-\int_0^t\left[\int_0^{t_1}(f(x-t_2))'dt_2\right]dt_1=2f(x)t+\int_0^t\int_0^{t_1}
[f(x+t_2)-f(x-t_2)]'dt_2dt_1\\
&=2f(x)t+\int_0^t\int_0^{t_1}\int_{x-t_2}^{x+t_2}f''(t_3)dt_3dt_2dt_1.
\end{align*}
\end{proof}

\begin{proof}[Proof of \thmref{thm3.7}]

Set
$$\eta(x)=\frac{1-\delta(x)}{\sqrt{q_1(x)}},\qquad \delta(x)=2(\varkappa_1(x)+\varkappa_2(x)),
\qquad |x|\gg1.$$
Then by \eqref{3.4}, \eqref{3.6}, \eqref{3.7}, \eqref{3.8}, \eqref{3.9}, \eqref{3.10} and \eqref{3.11}, we have
\begin{align*}
F(\eta(x))&=\int_0^{\sqrt2\eta(x)}\int_{x-t}^{x+t}q_1(\xi)d\xi dt+\int_0^{\sqrt2\eta(x)}\int_{x-t}^{x+t}q_2(\xi)d\xi dt\\
&\le \int_0^{\sqrt 2\eta(x)}\left[2q_1(x)t+\int_0^t\int_0^{t_1}\int_{x-t_2}^{x+t_2}q_1''(t_3)dt_3dt_2dt_1\right]\\
&\quad +\sqrt 2\eta(x)\sup_{t\in A(x)}\left|\int_{x-t}^{x+t}q_2(\xi)d\xi\right|\le\left(\sqrt 2\eta(x)\right)^2q_1(x)\\
&\quad +\frac{\left(\sqrt 2\eta(x)\right)^3}{6}\sup_{t_2\in A(x)}\left|\int_{x-{t_2}}^{x+{t_2}}q_1''(\xi)d\xi\right|+\sqrt 2(1-\delta(x))\varkappa_2(x)\\
&\le 2(1-\delta(x))^2+\frac{\sqrt 2}{3}(1-\delta(x))^3\varkappa_1(x)+\sqrt 2\varkappa_2(x)\\
&\le 2[(1-\delta(x))^2+\varkappa_1(x)+\varkappa_2(x)]\\
&=2\left[1-\frac{\delta(x)}{2}-\left(\frac{\delta
(x)}{2}-\delta^2(x)\right)-\varkappa_1(x)-\varkappa_2(x)\right]\le 2.
\end{align*}

Hence $d(x)\ge \eta(x)$ for $|x|\gg1$ by \lemref{lem3.5}.
Let now
$$\eta(x)=\frac{1+\delta(x)}{\sqrt{q_1(x)}},\qquad \delta(x)=2(\varkappa_1(x)+\varkappa_2(x)),\qquad |x|\gg1.$$
Then by the same arguments we obtain:
\begin{align*}
F(\eta(x))&=\int_0^{\sqrt 2\eta(x)}\int_{x-t}^{x+t}q_1(\xi)d\xi dt+\int_0^{\sqrt 2\eta(x)}
\int_{x-t}^{x+t}q_2(\xi)d\xi dt\\
&\ge \int_0^{\sqrt \eta(x)}\int_{x-t}^{x+t}\left[2q_1(x)t+\int_0^t\int_0^{t_1}\int_{x-t_2}^{x+t_2}q_1''(t_3)
dt_3dt_2dt_1\right]dt\\
&\quad -\sqrt2 \eta(x)\sup_{t\in A(x)}\left|\int_{x-t}^{x+t}q_2(\xi)d\xi\right|\ge\left(\sqrt 2\eta(x)\right)^2q_1(x)\\
&\quad -\frac{\left(\sqrt 2\eta(x)\right)^3}{6}\sup_{t_2\in A(x)}\left|\int_{x-t_2}^{x+t_2}q_1''(t_3)dt_3\right|-\sqrt 2(1+\delta(x))\varkappa_2(x)\\
&\ge 2(1+\delta(x))^2-\frac{\sqrt 2}{3}(1+\delta(x))^3\varkappa_1(x)-2\varkappa_2(x)\\
&\ge 2(1+\delta(x))+\varkappa_1(x)+\varkappa_2(x)\ge 2.
\end{align*}

Hence $d(x)\le\eta(x)$ for $|x|\gg1$ by \lemref{lem3.5}, and equality \eqref{3.12} is proven. Further, since the function $d(x)\sqrt{q_1(x)}$ is continuous and positive for all $x\in\mathbb R,$ for all $x_0\in (0,\infty)$ we the inequalities:
\begin{gather*}
0<m\le f(x)\le M<\infty,\quad |x|\le x_0\\
 m=\min_{|x|\le x_0}f(x),\quad M=\max_{|x|\le x_0} f(x),\quad f(x)=d(x)\sqrt{q_1(x)}.
 \end{gather*}
Together with \eqref{3.12}, this implies \eqref{3.13}.
\end{proof}

\begin{proof}[Proof of \thmref{thm4.1}]

Assume the contrary. Then \eqref{4.1} holds, the pair $\{L_{p,\mu};L_{p,\theta}\}$ is admissible for \eqref{1.1}, and there exists $x_0\in\mathbb R$ such that one of inequalities \eqref{2.1}, say, the second one, does not hold:
\begin{equation}\label{5.1}
\int_{x_0}^\infty q(t)dt=0\quad \Rightarrow\quad q(x)\equiv 0,\qquad x\in[x_0,\infty).
\end{equation}
Without loss of generality, in what follows we assume $x_0\ge 1.$ Let us introduce the functions $\varphi$ and $f_0.$
\begin{align}
&1)\quad\varphi\in C^\infty(\mathbb R),\quad \supp\varphi=[x_0,\infty),\quad 0\le\varphi(x)\le 1\quad for \quad x\in\mathbb R,\label{5.2}\\
&\qquad\qquad\qquad\qquad \varphi(x)\equiv 1\quad \text{for}\quad x\ge x_0+1\label{5.3}\\
&2)\quad f_0(x):=-\varphi''(x)+q(x)\varphi(x),\quad x\in\mathbb R.\label{5.4}
\end{align}

{}From 1)--2) we obtain the equality
\begin{gather}q(x)\varphi(x)\equiv0,\qquad x\in\mathbb R\quad\Rightarrow\nonumber\\
f_0(x)=-\varphi''(x),\quad x\in\mathbb R\quad\Rightarrow \quad \supp f_0=[x_0,x_0+1].\label{5.5}
\end{gather}
According to \eqref{5.5}, we conclude that $f_0\in L_{p,\theta}:$
$$\|f_0\|_{L_{p,\theta}}^p=\int_{-\infty}^\infty |\theta(x)f_0(x)|^pdx=\int_{x_0}^{x_0+1}|\theta(x)\varphi''(x)|^pdx=c(x_0)<\infty.$$
Since the pair $\{L_{p,\mu};L_{p,\theta}\}$ is admissible for \eqref{1.1}, we conclude that \eqref{1.1} for $f=f_0$ has a unique solution $y_0\in L_{p,\mu}.$ Then (see \eqref{5.4} and \eqref{5.5})
\begin{equation}\label{5.6}
y_0(x)=\varphi(x)+z(x),\qquad x\in\mathbb R,
\end{equation}
where $z(x),$ $x\in\mathbb R,$ is some soluton of \eqref{2.4}. From \eqref{2.4} and \eqref{5.1}, we obtain the equality
\begin{equation}\label{5.7}
z''(x)=0\quad\text{for}\quad x\in [x_0,\infty)\quad \Rightarrow\quad z(x)=c_1+c_2x\quad\text{for}\quad x\ge x_0.
\end{equation}
Let us show that $c_2=0.$ Assume to the contrary that $c_2\ne0. $ Choose $x_1$ so that to have the inequality
\begin{equation}\label{5.8}
\frac{|1+c_1|}{|c_2|}\cdot\frac{1}{x}\le\frac{1}{2}\quad\text{for}\quad x\ge x_1\ge x_0+1.
\end{equation}
Then (see \eqref{5.3})
\begin{align*}
\infty&>\|y_0\|_{p,\mu}^p\ge\int_{x_1}^\infty \mu(x)^p|\varphi(x)+z(x)|^pdx=\int_{x_1}^\infty\mu(x)^p|1+c_1+c_2x|^pdx\\
&\ge |c_2x_1|^p\int_{x_1}^\infty \mu(x)^p\left|1-\left|\frac{1+c_1}{c_2}\right|\frac{1}{x} \right|^pdx\ge\left|\frac{c_2x_1}{2}\right|^p\int_{x_1}^\infty\mu(x)^pdx=\infty,
\end{align*}
and we get a contradiction. Hence $c_2=0.$ Let us check that also $c_1=0.$ Assume that $c_1\ne0.$ Since $\varphi\in C^\infty(\mathbb R),$ from \eqref{5.2} it follows that $\varphi(x_0)=\varphi'(x_0)=0$ and therefore (see \eqref{5.7}):
\begin{align*}
y(x_0)&=\varphi(x_0)+z(x_0)=c_1,\\
y'(x_0)&=\varphi'(x_0)+z'(x_0)=0.
\end{align*}

In addition, $\varphi(x)\equiv0$ for $x\le x_0,$ and therefore from \eqref{5.5} and \eqref{5.6} it follows that the function $z$ is a solution of the Cauchy problem
 \begin{numcases}
{} z''(x)=q(x)z(x),\qquad\qquad x\le x_0\label{5.9}\\
 z(x_0)=c_1,\quad z'(x_0)=0.\label{5.10}
\end{numcases}
 Further, without loss of generality, we assume that $c_1=1.$
 Let us check that then we have the inequality
 \begin{equation}\label{5.11}
 z(x)\ge1\qquad\text{for}\qquad x\le x_0.
 \end{equation}
 Towards this end, first note that since $z(x_0)=1,$ we have $z(x)>0$ in some left half-neighborhood of the point $x_0$ (i.e., for $x\in(x_0-\varepsilon,x_0]$ for some $\varepsilon>0).$ But then $z(x)>0$ for all $x<x_0.$ Indeed, if this is not the case, then $z(x)$ has at least one zero on $(-\infty,x_0).$ Let $\tilde x$ be the first zero of $z(x)$ to the left from $x_0.$ Then $z'(\tilde x)\ge0.$ Indeed, if $z'(\tilde x)<0$ and $z(\tilde x)=0,$ then $z(x)<0$ in some right half-neighborhood of the $\tilde x.$ But $z(x_0)=1$ and $\tilde x<x_0.$ Hence, the interval $(\tilde x,x_0)$ contains a zero of $z(x),$ contrary to the definition of the point $\tilde x.$ Thus $z'(\tilde x)\ge0.$ On the other hand,
  \begin{align*}
  &z'(x_0)-z'(\tilde x)=\int_{\tilde x}^{x_0}q(\xi)z(\xi)d\xi\qquad\Rightarrow\\
  &z'(\tilde x)=-\int_{\tilde x}^{x_0}q(\xi)z(\xi)\le 0.
  \end{align*}

 Hence $z'(\tilde x)=0.$ But then the function $z(x)$ is a solution of the Cauchy problem
 \begin{align*} z''(x)=q(x)z(x),\quad x\le x_0 \\
z(\tilde x)=z'(\tilde x)=0
\end{align*}
 $$\qquad\Rightarrow\qquad z(x)\equiv 0,\quad x\le x_0.$$

 We get a contradiction because $z(x_0)=1.$ Thus $z(x)>0$ for $x\le x_0.$ Then for $x\le x_0$, we have
 $$-z'(x)=z'(x_0)-z'(x)=\int_x^{x_0}q(\xi)z(\xi)d\xi\ge0\quad\Rightarrow \quad z'(x)\le 0,\quad x\le x_0.$$
 Hence $z(x)\ge z(x_0)=1$ for $x\le x_0.$ This implies that
 \begin{align*}
 \infty&>\|y_0\|_{p,\mu}=\int_{-\infty}^\infty |\mu(x)y_0(x)|^pdx\ge\int_{-\infty}^{x_0}|\mu(x)y_0(x)|^pdx\\
 &=\int_{-\infty}^{x_0}|\mu(x)z(x)|^pdx\ge\int_{-\infty}^{x_0}\mu(x)^pdx=\infty.
 \end{align*}
 We get a contradiction. Hence $c_1=0,$ and we obtain the equality
 $$y_0(x)=\varphi(x),\qquad x\ge x_0 \qquad \Rightarrow$$
 \begin{align*}
 \infty&>\|y_0\|_{p,\mu}^p\ge\int_{x_0+1}^\infty|\mu(x)y_0(x)|^pdx=\int_{x_0+1}^\infty|\mu(x)\varphi
 (x)|^pdx\\
 &=\int_{x_0+1}^\infty\mu(x)^pdx=\infty.
 \end{align*}
 We get a contradiction. Hence \eqref{5.1} does not hold.
\end{proof}

\begin{proof}[Proof of \lemref{lem4.2}]

Let us show that in the  case of \eqref{4.2} for all $p\in[1,\infty)$ we have the equalities
\begin{equation}\label{5.12}
\int_{-\infty}^0(\mu(t)u(t))^pdt=\int_0^\infty(\mu(t)v(t)^pdt=\infty.
\end{equation}
We only consider the second equality because the first one can be proved in the same way. For $p=1$ equality \eqref{5.11} follows from \thmref{thm2.2} and \eqref{4.2} is a straightforward manner. Let $p\in(1,\infty)$, $p'=p(p-1)^{-1}.$ The following relations rely only on \thmref{thm2.2}:
\begin{align}
\int_0^\infty&\frac{dt}{v(t)^{p'}} =\int_0^\infty\frac{v'(t)v(t)^{-p'}}{v'(t)}dt\le \frac{1}{v'(0)}\int_0^\infty v'(t)v(t)^{-p'}dt\nonumber\\
&=\frac{1}{p'-1}\frac{1}{v'(0)}\left(\frac{1}{v(0)^{p'-1}}-\frac{1}{v(\infty)^{p'-1}}\right)
\le \frac{1}{p'-1}\frac{1}{v'(0)v(0)^{p'-1}}=c(p)<\infty.\label{5.13}
\end{align}

Let $A>0.$ Below we use H\"older's inequality and \eqref{5.13}:
$$\int_0^A\mu(t)dt\le\left[\int_0^A(\mu(t)v(t))^pdt\right]^{1/p}\cdot\left[\int_0^A\frac{dt}{v(t)p'}\right]
^{1/p'}\le c(p)\left[\int_0^A (\mu(t)v(t))^pdt\right]^{1/p}.$$
Now, to obtain \eqref{5.12}, in the last inequality we let $A$ tend to infinity. Let us now go over to the proof of the lemma. By \thmref{thm2.2}, the general solution of \eqref{2.4} is of the form
$$z(x)=c_1u(x)+c_2v(x),\qquad x\in\mathbb R.$$
Let $z\in L_{p,\mu}.$ Then $c_2=0$. Indeed, if $c_2\ne0,$ then denote $x_1\gg1,$ a number such that for all $x\ge x_1$ we have the inequality (see \eqref{2.7}):
\begin{equation}\label{5.14}
\left|\frac{c_1}{c_2}\right|\frac{u(x)}{v(x)}\le\frac{1}{2},\qquad x\ge x_1.
\end{equation}
Now from \eqref{5.12}, \eqref{5.14} and \thmref{thm2.2} it follows that
\begin{align*}
\infty&>\|z\|_{p,\mu}^p=\int_{-\infty}^\infty|\mu(x)(c_1u(x)+c_2v(x)|^pdx\\
&\ge|c_2 |^p\int_{x_1}^\infty(\mu(x)v(x))^p\left| 1-\left|\frac{c_1}{c_2}\right|\, 
\frac{u(x)}{v(x)}\right|^pdx\ge\left|\frac{c_2}{2}\right|^p\int_{x_1}^\infty(\mu(x)v(x))^pdx=\infty.
\end{align*}
We get a contradiction. Hence $c_2=0.$ The equality $c_1=0$ now follows from \eqref{2.5} and \eqref{5.12}.
\end{proof}

\begin{proof}[Proof of \thmref{thm4.3} for $p\in(1,\infty)$] \textit{Necessity}.

 We need the following lemma.

 \begin{lem}\label{lem5.1}
 Let $p\in[1,\infty)$. Suppose that conditions \eqref{4.2} hold, and the pair $\{L_{p,\mu};L_{p,\theta}\}$ is admissible for \eqref{1.1}. Then, if $f\in L_p$ and $\supp f=[x_1,x_2],$ $x_2-x_1<\infty,$ then $f\in L_{p,\theta}$ and the solution $y\in L_{p,\mu}$ of \eqref{1.1} which corresponds to $f$ is of the form \eqref{2.13}.
 \end{lem}

\renewcommand{\qedsymbol}{}
 \begin{proof}
 Below we only consider the case $p\in (1,\infty)$ (for $p=1$ the arguments are similar). Let us continue the function $f$ by zero beyond the segment $[x_1,x_2]$ and maintain the original notation. From the obvious inequalities
 \begin{equation}\label{5.15}
 c^{-1}\le\theta(x)\le c,\quad x\in[x_1,x_2],\quad c=c(x_1,x_2),
 \end{equation}
 it follows that $f\in L_{p,\theta}.$ Set (see \eqref{2.9}, \eqref{2.13})
 \begin{align}
 \tilde y(x)&=\int_{-\infty}^\infty G(x,t)f(t)dt\nonumber\\
 &=u(x)\int_{-\infty}^xv(t)f(t)dt+v(x)\int_x^\infty u(t)f(t)dt,\quad x\in\mathbb R.\label{5.16}
 \end{align}

 Let us estimate the integrals in \eqref{5.16}:
 \begin{align}
 \int_{-\infty}^xv(t)|f(t)|dt&\le\left[\int_{x_1}^{x_2}\left(\frac{v(t)}{\theta(t)}\right)^{p'}dt
 \right]^{1/p'}\cdot\left[\int_{x_1}^{x_2}|\theta(t)f(t)|^pdt\right]^{1/p}\nonumber\\
 &\le c\left(\int_{x_1}^{x_2}v(t)^{p'}dt\right)^{1/p'}\cdot\|f\|_{p,\theta},\quad x\in\mathbb R,\label{5.17}\\
   \int_{x}^\infty u(t)|f(t)|dt&\le\left[\int_{x_1}^{x_2}\left(\frac{u(t)}{\theta(t)}\right)^{p'}dt
 \right]^{1/p'}\cdot\left[\int_{x_1}^{x_2}|\theta(t)f(t)|^pdt\right]^{1/p}\nonumber\\
 &\le c\left(\int_{x_1}^{x_2}u(t)^{p'}dt\right)^{1/p'}\cdot\|f\|_{p,\theta},\quad x\in\mathbb R.\label{5.18}.
 \end{align}

 {}From \eqref{5.17} and \eqref{5.18} it follows that the function $\tilde y(x),$ $s\in\mathbb R,$ is well-defined. It is also easy to see that the function $\tilde y(x),$ $x\in\mathbb R$ is a particular solution of \eqref{1.1}. But, since $f\in L_{p,\theta},$ \eqref{1.1} has a unique solution $y\in L_{p,\theta}.$ This means that we have the equality
 $$y(x)=\tilde y(x)+c_1u(x)+c_2v(x),\qquad x\in\mathbb R.$$

 Let us check that $c_1=c_2=0.$ Assume, say, that
$c_2\ne0.$ Then for $x\ge x_2$, we get
\begin{align*}
|y(x)|&\ge |c_2|v(x)-|c_1|u(x)-u(x)\int_{x_1}^{x_2}v(t)|f(t)|dt\\
&=|c_1|v(x)\left[1-\left|\frac{c_1}{c_2}\right|\frac{u(x)}{v(x)}-\frac{u(x)}{v(x)}\int_{x_1}^{x_2}
v(t)|f(t)|dt\right].
\end{align*}

{}From \eqref{2.7} and \eqref{5.17} it follows that there exists $x_3\ge\max\{1,x_2\}$ such that
$$|y(x)|\ge\frac{1}{2}|c_2|v(x)\qquad\text{for}\quad x\ge x_3\qquad\Rightarrow \qquad \text{(see \eqref{5.12}):}$$
$$\infty>\|y\|_{p,\mu}^p\ge\int_{x_1}^\infty|\mu(x)y(x)|^pdx\ge\left|\frac{c_2}{2}\right|^p\int_{x_3}^\infty
|\mu(x)v(x)|^pdx=\infty.$$
We get a contradiction. Hence $c_2=0.$ Similarly, we prove that also $c_1=0,$ and therefore $y=\tilde y$ (see \eqref{5.16}). Let $[x_1,x_2]$ be any finite segment. Set
\begin{equation}\label{5.19}
f(t)=\begin{cases} \theta(t)^{-p'}\cdot u(t)^{p'-1},&\quad t\in [x_1,x_2]\\
0,&\quad t\notin [x_1,x_2]\end{cases}
\end{equation}
Then
\begin{equation}\label{5.20}
\|f\|_{L_{p,\theta}}^p=\int_{x_1}^{x_2}|\theta(t)f(t)|^pdt=\int_{x_1}^{x_2}\frac{\theta(t)^pu^{p(p'-1)}(t)}{\theta(t)^{p'p}} dt=\int_{x_1}^{x_2}\left(\frac{u(t)}{\theta(t)}\right)^{p'}dt<\infty.
\end{equation}
Therefore, since the pair $\{L_{p,\mu};L_{p,\theta}\}$ is admissible for \eqref{1.1}, in the case of \eqref{5.19} equation \eqref{1.1} has a solution $y\in L_{p,\mu}.$ This solution is of the form \eqref{2.13} (see \lemref{lem5.1}). This implies that
\begin{align}
&\infty>\|y\|_{p,\mu}^p=\int_{-\infty}^\infty|\mu(x)y(x)|^pdx\nonumber\\
&=\left\{\int_ {-\infty}^\infty\mu(x)^p\left[u(x)\int_ {-\infty}^x v(t)f(t)dt+v(x) \int_x^\infty u(t)f(t)dt\right]^pdx\right\}\nonumber\\
 &\ge\int_{-\infty}^\infty(\mu(x)v(x))^p\left(\int_x^\infty u(t)f(t)dt\right)^pdx\ge\int_{-\infty}^{x_1}(\mu(x)v(x))^p\left(\int_x^\infty
 u(t)f(t)dt\right)^pdx\nonumber\\
 &\ge\int_{-\infty}^{x_1}(\mu(x)v(x))^pdx \left(\int_{x_1}^{x_2}u(t)f(t)dt\right)^p=\int_{-\infty}^{x_1}(\mu(x)v(x))^pdx
 \left(\int_{x_1}^{x_2}\left(\frac{u(t)}{\theta(t)}\right)^{p'}dt\right)^p\label{5.21}
\end{align}

Now, using \eqref{5.21}, \eqref{5.20} and \eqref{1.5}, we obtain
\begin{align*}
\left[\int_{-\infty}^{x_1}(\mu(x)v(x))^pdx\right]^{1/p}\int_{x_1}^{x_2}\left(\frac{u(t)}{\theta(t)}\right)^{p'}dt&\le
\|y\|_{p,\mu}\le c(p)\|f\|_{p,\theta}\\
&=c(p)\left[\int_{x_1}^{x_2}\left(\frac{u(t)}{\theta(t)}\right)^{p'}dt\right]^{1/p}\quad\Rightarrow\\
\end{align*}
$$\left(\int_{-\infty}^{x_1}(\mu(t)v(t))^pdt\right)^{1/p}\left(\int_{x_1}^{x_2}\left(\frac{u(t)}{\theta(t)}\right)^{p'}dt
\right)^{1/p'}\le c(p)<\infty.$$
Since in this inequality $x_1$ and $x_2$ $(x_1\le x_2)$ are arbitrary numbers, we conclude that
$$M=\sup_{x\in\mathbb R}\left(\int_{-\infty}^{x}(\mu(t)v(t))^pdt\right)^{1/p}\cdot\left(\int_x^\infty\left( \frac{u(t)}{\theta(t)}\right)^{p'}dt\right)^{1/p'}\le c(p)<\infty.$$
This inequality means that the operator $S_2:L_p\to L_p,$
\begin{equation}\label{5.22}
(S_2f)(x)=\mu(x)v(x)\int_x^\infty\frac{u(t)}{\theta(t)}f(t)dt,\quad x\in\mathbb R
\end{equation}
is bounded (see \thmref{thm2.7}). Similarly, we use \thmref{thm2.8} to conclude that the operator $S_1:L_p\to L_p,$
\begin{equation}\label{5.23}
(S_1f)(x)=\mu(x)u(x)\int_{-\infty}^x\frac{v(t)}{\theta(t)}f(t)dt,\quad x\in\mathbb R
\end{equation}
is bounded. Since we have the equality (see \eqref{2.9} and \eqref{4.3})
\begin{equation}\label{5.24}
S=S_1+S_2
\end{equation}
our assertion now follows from the triangle inequality for norms.
 \end{proof}
 \end{proof}

\newpage

\begin{proof}[Proof of \thmref{thm4.3}] \textit{Sufficiency.}
\renewcommand{\qedsymbol}{\openbox}
\begin{lem}\label{lem5.2}
Let $p\in[1,\infty),$ and let $S,$ $S_1,$ $S_2$ be operators \eqref{4.3}, \eqref{5.23} and \eqref{5.22}, respectively.
Then we have the inequalities
\begin{equation}\label{5.25}
\frac{\|S_1\|_{p\to p}+\|S_2\|_{p\to p}}{2}\le \|S\|_{p\to p}\le\|S_1\|_{p\to p}+\|S_2\|_{p\to p}.
\end{equation}
\end{lem}

 \begin{proof}
 The upper estimate in \eqref{5.25} follows from \eqref{5.24}. To prove the lower estimate in \eqref{5.25}, we use the following obvious relations:
 \begin{align*}
 \|S_1(f)\|_p^p&=\int_{-\infty}^\infty \mu(x)^p\left|u(x)\int_{-\infty}^x\frac{v(t)}{\theta(t)}f(t)dt\right|^pdx\\
 &\le
\int_{-\infty}^\infty\mu(x)^p\left(u(x)\int_{-\infty}^x\frac{v(t)}{\theta(t)}|f(t)|dt
\right)^pdx\\
&\le\int_{-\infty}^\infty\mu(x)^p\left[u(x)\int_{-\infty}^x\frac{v(t)}{\theta(t)}|f(t)|dt+v(x)\int_x^\infty\frac{u(t)}{\theta
(t)}|f(t)|dt\right]^pdx\\
&=\int_{-\infty}^\infty\left|\mu(x)\int_{-\infty}^\infty\frac{G(x,t)}{\theta(t)}|f(t)|
dt\right|^pdx=\|S(|f|)\|_p^p
\le\|S\|_{p\to p}^p\cdot\|f\|_p^p.
\end{align*}

This implies that $\|S_1\|_{p\to p}\le\|S\|_{p\to p}.$ Similarly, we check that $\|S_2\|_{p\to p}\le\|S\|_{p\to p}.$ These inequalities imply the lower estimate in \eqref{5.25}.
\end{proof}

Let us now go over to the proof of the theorem. Since \eqref{2.1} holds, equation \eqref{2.4} has a FSS $\{u,v\}$ with the properties from \thmref{thm2.2}. Since the operator $S:L_p\to L_p $ is bounded,  so are also the operators $S_i: L_p\to L_p,$ $i=1,2$ (see \eqref{5.25}). Then, by Theorems \thmref{thm2.7} and \thmref{thm2.8}, we obtain the inequalities
\begin{align}
\tilde M_p&\doe \sup_{x\in\mathbb R}\left(\int_{-\infty}^x\left(\frac{v(t)}{\theta(t)}\right)^{p'}dt\right)^{1/p'}\cdot\left(\int_x^\infty(\mu(t)u(t)^pdt\right) ^{1/p}<\infty,\label{5.26}\\
 M_p&\doe \sup_{x\in\mathbb R} \left(\int_{-\infty}^x(\mu(t)v(t)^pdt\right) ^{1/p} \cdot \left(\int^{ \infty}_x\left(\frac{u(t)}{\theta(t)}\right)^{p'}dt\right)^{1/p'}<\infty.\label{5.27}
\end{align}
These inequalities imply that the function
\begin{equation}\label{5.28}
y(x)=(Gf)(x)=u(x)\int_{-\infty}^x v(t)f(t)dt+v(x)\int_x^\infty u(t)f(t)dt,\quad x\in\mathbb R
\end{equation}
is well-defined because the integrals in \eqref{5.28} converge:
\begin{align*}
\int_{-\infty}^x v(t)|f(t)|dt&\le\left(\int_{-\infty}^x\left(\frac{v(t)}{\theta(t)}\right)^{p'}dt\right)^{1/p'}\cdot\|f\|_{p,\theta},\quad x\in\mathbb R, \\
\int^{\infty}_x u(t)|f(t)|dt&\le\left(\int_{-\infty}^x\left(\frac{u(t)}{\theta(t)}\right)^{p'}dt\right)^{1/p'}\cdot\|f\|_{p,\theta},\quad x\in\mathbb R.\\
\end{align*}

Further, one can check in a straightforward manner (see \thmref{thm2.2}) that the function $y(x),$ $x\in\mathbb R$ is a solution of \eqref{1.1}. In addition,
\begin{align*}
\|y\|_{p,\mu}&=\left[\int_{-\infty}^\infty\left(\mu(x)\left|\int_{-\infty}^\infty G(x,t)f(t)dt\right|\right)^pdx\right]^{1/p}\\
&=\left[\int_{-\infty}^\infty\left(\mu(x)\left|\int_{-\infty}^\infty\frac{G(x,t)}{\theta(t)}(\theta(t)f(t))dt\right|\right)^p
dx\right]^{1/p}\\
&=\|S(\theta f)\|_p\le\|S\|_{p\to p}\cdot\|\theta f\|_p=\|S\|_{p\to p}\cdot\|f\|_{p,\theta},
\end{align*}
i.e., \eqref{1.5} holds. It only remains to refer to \lemref{lem4.2}.
 \end{proof}

 \renewcommand{\qedsymbol}{}
\begin{proof}[Proof of \thmref{thm4.3} for $p=1$] \textit{Necessity}.

Let $[x_1,x_2]$ be an arbitrary finite segment, and let $f\in L_1$ be such that $\supp f=[x_1,x_2].$ Then (see \eqref{5.15}) $f\in L_{1,\theta}$ and therefore equation \eqref{1.1} with such a right-hand side has a unique solution $y\in L_{1,\mu}.$ By \lemref{lem5.1}, this solution is given by formula \eqref{2.13} and satisfies \eqref{1.5}. Let us introduce  the operator $\tilde S:$
$$(\tilde Sg)(x)=\mu(x)\int_{x_1}^{x_2}\frac{G(x,t)}{\theta(t)}g(t)dt,\quad x\in[x_1,x_2],\quad g\in L_1(x_1,x_2)$$
and the function $g$ given on the sequence $[x_1,x_2]$ by the formula
$$g(x)=\theta(x)f(x),\qquad x\in [x_1,x_2].$$
Then we have
\begin{align*}
\|\tilde Sg\|_{L_1(x_1,x_2)}&=\int_{x_1}^{x_2}\left|\mu(x)\int_{x_1}^{x_2}\frac{G(x,t)}{\theta(t)}g(t)dt\right|dx\\
&=\int_{x_1}^{x_2}\mu(x)\left|\int_{x_1}^{x_2}G(x,t)f(t)dt\right|dx=\int_{x_1}^{x_2}\mu(x)\left|\int_{-\infty}^\infty G(x,t)f(t)dt\right|dx\\
&=\int_{x_1}^{x_2}\mu(x)|y(x)|dx\le\int_{-\infty}^\infty \mu(x)|y(x)|dx=\|y\|_{1,\mu}\le c(1)\|f\|_{1,\theta}\\
&=c(1)\int_{-\infty}^\infty\theta(t)|f(t)|dt=c(1)\int_{x_1}^{x_2}|\theta(t)f(t)|dt=c(1)\|g\|_{L_1(x_1,x_2)}.
\end{align*}

Together with \eqref{2.22} and \eqref{2.9}, this implies that
\begin{align*}
&\|\tilde S\|_{L_1(x_1,x_2)\to L_1(x_1,x_2)}\le c(1)\qquad \Rightarrow\\
&\sup_{x\in[x_1,x_2]}\frac{1}{\theta(x)}\int_{x_1}^{x_2}\mu(t)G(x,t)dt=\|\tilde S\|_{L_1(x_1,x_2)\to L_1(x_1,x_2)}\le c(1).
\end{align*}
In the last inequality, $x_1$ and $x_2$ are arbitrary numbers. Hence
$$\sup_{x\in\mathbb R}\frac{1}{\theta(x)}\int_{-\infty}^\infty \mu(t)G(x,t)dt\le c(1)<\infty.$$
But then by \thmref{thm2.9} we obtain that $\|S\|_{L_1\to L_2}\le c(1)<\infty,$ as required.
\end{proof}

\renewcommand{\qedsymbol}{\openbox}

\begin{proof}[Proof of \thmref{thm4.3} for $p=1$] \textit{Sufficiency.}

{}From \eqref{2.1} it follows that equation \eqref{2.4} has a FSS $\{u,v\}$ (see \thmref{thm2.2}), the Green function and the operator $S$ are defined  (see \eqref{2.9} and \eqref{4.3}). Further, the operators $S_i,$ $i=1,2$ (see \eqref{4.4}, \eqref{4.5}) are bounded because so is the operator $S:L_1\to L_1$ (see \lemref{lem5.2}). Let now $f\in L_{1,\theta}$ and $g=\theta\cdot|f|.$ Then $0\le g\in L_1,$ $S_ig\in L_1,$ $i=1,2,$ and one has the inequalities
\begin{equation}\label{5.29}
0\le (S_ig)(x)<\infty,\qquad \forall x\in\mathbb R,\qquad i=1,2.
\end{equation}

We will prove \eqref{5.29} for $i=1$ (the case $i=2$ is considered in a similar way). Assume to the contrary that there exists $x_1\in\mathbb R$ such that $(S_1g)(x_1)=\infty.$ Let $x_2>x_1.$ Then, since the functions $\mu$ and $u$ are continuous, we have
\begin{align*}
(S_1g)(x_2)&=\mu(x_2)u(x_2)\int_{-\infty}^{x_2}\frac{v(t)}{\theta(t)}g(t)dt\\
&\ge \frac{\mu(x_2)u(x_2)}{\mu(x_1)u(x_1)}\left[\mu(x_1)u(x_1)\int_{-\infty}^{x_1}\frac{v(t)}{\theta(t)}g(t)dt\right]=
\frac{\mu(x_2)u(x_2)}{\mu(x_1)u(x_1)}(S_1g)(x_1)=\infty\\
\Rightarrow\qquad\qquad\qquad\\
\infty&>\|Sg\|_1=\int_{-\infty}^\infty\mu(x)u(x)\left|\int_{-\infty}^x\frac{v(t)}{\theta(t)}g(t)dt\right|dx\\
&\ge\int_{x_1}^\infty \mu(x)u(x)\left(\int_{-\infty}^x\frac{v(t)}
{\theta(t)}g(t)dt\right)dx=\int_{x_1}^\infty(S_1g)(x)dx=\infty.
\end{align*}
We get a contradiction. Hence, inegualities  \eqref{5.29} hold. From \eqref{5.29} and the definition of $g$ we obtain
\begin{equation}\label{5.30}
\int_{-\infty}^xv(t)|f(t)|dt<\infty,\qquad \int_x^\infty u(t)|f(t)|dt<\infty \qquad \forall x\in\mathbb R.
\end{equation}
For instance,
\begin{align*}
\int_{-\infty}^x v(t)|f(t)|dt&=\frac{1}{\mu(x)u(x)}\left[\mu(x)u(x)\int_{-\infty}^x \frac{v(t)}{\theta(t)}\cdot(\theta(t) |f(t)|)dt\right]\\
&=\frac{1}{\mu(x)u(x)}(S_1g)(x)<\infty\quad\Rightarrow\quad\eqref{5.30}
\end{align*}

Thus, if $f\in L_{1,\theta},$ then by \eqref{5.30} the following integrals converge:
$$\int_{-\infty}^x v(t)f(t)dt,\qquad \int_x^\infty u(t)f(t)dt,\qquad x\in\mathbb R$$
and therefore, for $x\in\mathbb R,$ the function
$$y(x)=(Gf)(x)=u(x)\int_{-\infty}^x v(t)f(t)dt+v(x)\int_x^\infty u(t)f(t)dt,\quad x\in\mathbb R$$
is well-defined. This immediately implies that $y(x)$ is a solution of \eqref{1.1}. In addition, \eqref{1.5} holds:
\begin{align*}
\|\mu y\|_1&=\int_{-\infty}^\infty \mu(x)\left|\int_{-\infty}^\infty \frac{G(x,t)}{\theta(t)}(\theta(t)f(t))dt\right|dx \le \int_{-\infty}^\infty \mu(x)\int_{-\infty}^\infty\frac{G(x,t)}{\theta(t)}|g(t)|dt\, d\\ &=\|Sg\|_1\le\|S\|_{1\to 1}\cdot \|g\|_1
 =\|S\|_{1\to 1}\cdot\|f\|_{1,\theta}\quad \Rightarrow\quad\eqref{1.5}.
\end{align*}
It remains to note that by \lemref{lem4.2} this solution is unique in the class $L_{1,\mu}.$
\end{proof}

\begin{proof}[Proof of \lemref{lem4.5}]

{}From \eqref{2.2} we obtain the inequality
$$2\le\sqrt 2 d(x)\int_{x-\sqrt 2 d(x)}^{x+\sqrt 2 d(x)}q(\xi)d\xi,\qquad x\in\mathbb R.$$
Together with the formula for $|d'(x)|$ (see the proof of \lemref{lem3.1}), this implies that
\begin{align*}
|d'(x)|&\le \frac{d(x)}{\sqrt 2}\left|\int_x^{x+\sqrt 2 d(x)}q(\xi)d\xi-\int_{x-\sqrt 2 d(x)}^x q(\xi)d\xi\right|\cdot\left(d(x)\int_{x-\sqrt 2 d(x)}^{x+\sqrt 2 d(x)}q(\xi)d\xi\right)^{-1} \\
&\le\frac{1}{2}d(x)\left|\int_x^{x+\sqrt 2 d(x)}q(\xi)d\xi-\int^x_{x-\sqrt 2 d(x)}\right|=\frac{\nu(x)}{2},\quad x\in\mathbb R\quad \Rightarrow
\end{align*}
\begin{equation}
\lim_{|x|\to \infty}d'(x)=0.\label{5.31}
\end{equation}

Let us now go to \eqref{4.8}. Fix $\varepsilon\in(0,1/\sqrt 2)$ (see \eqref{4.9} regarding the case $\varepsilon\ge 1/\sqrt 2$). Then there exists $x_0=x_0(\varepsilon)\gg 1$ such that we have the inequality (see \eqref{5.31}
\begin{equation}\label{5.32}
|d'(x)|\le\varepsilon\qquad\text{if}\qquad |x|\ge x_0.
\end{equation}
It is easy to see that all possible cases of placing the numbers $t,x\in\mathbb R$ and the segments $(-\infty, x_0],$ \ $[-x_0, x_0]$ and $[x_0, \infty]$ can be put in the following table:

 \begin{equation}\vbox{\offinterlineskip
\hrule
\hrule
\halign{&\vrule#&
\strut
#\tabskip=.09em\cr
height4pt&\omit&&\omit&&\omit  &\cr
&\qquad \qquad$1.1$
\hfil
&& \qquad \quad$1.2$\quad&&\qquad \qquad$ 1.3$
  \hfil &\cr
  &\quad $x\in(-\infty,-x_0]$\hfil&& \quad  $x\in(-\infty,-x_0]$\hfil&& \quad  $x\in(-\infty,-x_0]$\hfil &\cr
  &\quad $t\in(-\infty,-x_0]\hfil$&& \quad  $t\in[-x_0,x_0]$\hfil&& \quad   $t\in[x_0,\infty]$\hfil &\cr
 \noalign{\hrule}
 &\qquad \qquad$2.1$
\hfil
&& \qquad \quad$2.2$\quad&&\qquad \qquad$ 2.3$
  \hfil &\cr
  &\quad $x\in[-x_0,x_0]$\hfil&& \quad  $x\in[-x_0,x_0]$\hfil&& \quad  $x\in[-x_0,x_0]$\hfil &\cr
  &\quad $t\in(-\infty,-x_0]\hfil$&& \quad  $t\in[-x_0,x_0]$\hfil&& \quad   $x\in[x_0,\infty)$\hfil &\cr
 \noalign{\hrule}
 &\qquad \qquad$3.1$
\hfil
&& \qquad \quad$3.2$\quad&&\qquad \qquad$ 3.3$
  \hfil &\cr
  &\quad $x\in(x_0,\infty]$\hfil&& \quad  $x\in(x_0,\infty)$\hfil&& \quad  $x\in(x_0,\infty)$\hfil &\cr
  &\quad $t\in(-\infty,-x_0]\hfil$&& \quad  $t\in[-x_0,x_0]$\hfil&& \quad   $t\in[-x_0,\infty)$\hfil &\cr
 \noalign{\hrule}
} \hrule}
 \raisetag{3\baselineskip} \label{5.33}
 \end{equation}

 \bigskip
We check inequalities \eqref{4.8} separately in each case appearing in \eqref{5.33}.

 \newpage

\noindent \textit{Cases 1.1 and 3.3}.

Both cases are treated in the same way. Let us introduce the standing notation for the whole proof:
\begin{gather*}
m(\varepsilon)=\min_{t\in[-x_0,x_0]} d(t),\qquad M(\varepsilon)=\max_{t\in[-x_0,x_0]}d(t)\\
c(\varepsilon)=\max\left\{\frac{1}{m(\varepsilon)}, M(\varepsilon) \right\},\\
a=\min\{x,t),\qquad b=\max\{x,t\}.
\end{gather*}

Consider, say, Case 3.3. The following implications are obvious:
\begin{gather*}
-\varepsilon \le d'(\xi)\le\varepsilon\quad\text{for}\quad \xi\in[a,b]\Rightarrow -\frac{\varepsilon}{d(\xi)}\le \frac{d'(\varepsilon)}{d(\xi)}\le\frac{\varepsilon}{d(\xi)},\quad \xi\in[a,b]\\
 \Rightarrow - \varepsilon\left|\int_x^t\frac{d\xi}{d(\xi)}\right| =-\varepsilon\int_a^b\frac{d\xi}{d(\xi)}\le \ln\frac{d(b)}{d(a)}\le\varepsilon\int_a^b\frac{d\xi}{d(\xi)}=\varepsilon\left|\int_x^t\frac{d\xi}{d(\xi)}\right|\quad\Rightarrow\\
 \exp\left(-\varepsilon\left|\int_x^t\frac{d\xi}{d(\xi)}\right|\right) \le \frac{d(b)}{d(a)},\qquad \frac{d(a)}{d(b)}\le \exp\left(\varepsilon\left|\int_x^t\frac{d\xi}{d(\xi)}\right|\right)\quad\Rightarrow \quad \eqref{4.8}.
\end{gather*}

\noindent \textit{Cases 1.2 and 2.1}.

Both cases are treated in the same way.
For instance, in Case 1.2 we have
 \begin{align*}
 \frac{d(t)}{d(x)}&=\frac{d(t)}{d(-x_0)}\cdot \frac{d(-x_0)}{d(x)}\le c(\varepsilon)^2\exp\left(\varepsilon\left|\int_x^{-x_0}\frac{d\xi}{d(\xi)}\right|\right)\\
 &\le c(\varepsilon)^2\exp\left(\varepsilon\left|\int_x^{-x_0}\frac{d\xi}{d(\xi)}+\int_{-x_0}^t\frac{d\xi}{d(\xi)}\right|\right)=
 c(\varepsilon)^2\exp\left(\varepsilon\left|\int_x^t\frac{d\xi)}{d(\xi)}\right|\right);
 \end{align*}
\begin{align*}
\frac{d(t)}{d(x)}&=\frac{d(t)}{d(-x_0)}\cdot \frac{d(-x_0)}{d(x)}\ge c(\varepsilon)^{-2}\exp\left(-\varepsilon\left|\int_x^{-x_0}\frac{d\xi}{d(\xi)}\right|\right)\\
&\ge c(\varepsilon)^2\exp\left(-\varepsilon\left|\int_x^{-x_0}\frac{d\xi}{d(\xi)}+\int_{-x_0}^t\frac{d\xi}{d(\xi)}\right|\right)=
 c(\varepsilon)^2\exp\left(-\varepsilon\left|\int_x^t\frac{d\xi)}{d(\xi)}\right|\right)\quad\Rightarrow\quad\eqref{4.8}.
 \end{align*}

\noindent \textit{Cases 1.3 and 3.1}.

Both cases are treated in the same way.
For instance, in Case 1.3 we have
\begin{align*}
\frac{d(t)}{d(x)}&=\frac{d(-x_0)}{d(x)}\cdot \frac{d(x_0)}{d(-x_0)}\cdot\frac{d(t)}{d(x_0)}\le \frac{M}{m}\exp\left(\varepsilon\left|\int_x^{-x_0}\frac{d\xi}{d(\xi)}\right|+\varepsilon
\left|\int_{x_0}^t\frac{d\xi}{d(\xi)}\right|
\right)\\
&\le c(\varepsilon)^2\exp\left[\varepsilon\left(\int_x^{-x_0}\frac{d\xi}{d(\xi)}+\int_{-x_0}^x\frac{d\xi}{d(\xi)} \right)+\int_{x_0}^t\frac{d\xi}{d(\xi)}\right]\\
&=
c(\varepsilon)^2\exp\left(\varepsilon\left|\int_x^t\frac{d\xi)}{d(\xi)}\right|\right);
\end{align*}
 \begin{align*}
 \frac{d(t)}{d(x)}&=\frac{d(-x_0)}{d(x)}\cdot \frac{d(x_0)}{d(-x_0)} \cdot\frac{d(t)}{d(x_0)}\ge \frac{m}{M}\exp\left(-\varepsilon\left|\int_x^{-x_0}\frac{d\xi}{d(\xi)}\right|-\varepsilon
 \left|\int_{x_0} ^t\frac{d\xi}{d\xi}\right|\right)\\
 &\ge c(\varepsilon)^{-2}\exp\left[-\varepsilon\left(\int_x^{-x_0}\frac{d\xi}{d(\xi)}+\int_{-x_0}^{x_0}\frac{d\xi}{d(\xi)}
 +\int_{x_0}^t\frac{d\xi}{d(\xi)}\right)\right ]\ge
 c(\varepsilon)^{-2}\exp\left(-\varepsilon\left|\int_x^t\frac{d\xi)}{d(\xi)}\right|\right) .
 \end{align*}

\noindent \textit{Case 2.2}.

We have
\begin{align*}
\frac{d(t)}{d(x)}&\le \frac{M(\varepsilon)}{m(\varepsilon)}\le c(\varepsilon)^2\exp\left(\varepsilon\left|\int_x^t\frac{d\xi}{d(\xi)}\right|\right);\\
\frac{d(t)}{d(x)}&\ge \frac{m(\varepsilon)}{M(\varepsilon)}\ge c(\varepsilon)^{-2}\exp\left(-\varepsilon\left|\int_x^t\frac{d\xi}{d(\xi)}\right|\right).\\
\end{align*}

\noindent \textit{Cases 2.3 and 3.2}.

Both cases are treated in the same way.
For instance, in Case 2.3 we have
\begin{align*}
\frac{d(t)}{d(x)}&=\frac{d(x_0)}{d(x)}\cdot\frac{d(t)}{d(x_0)}\le \frac{M(\varepsilon)}{m(\varepsilon)} \exp\left(\varepsilon\left|\int_{x_0}^t\frac{d\xi}{d(\xi)}\right|\right) \\
 &\le c(\varepsilon)^2\exp\left[\varepsilon\left(\int_x^{x_0}\frac{d\xi}{d(\xi)}+\int_{x_0}^t\frac{d\xi}{d(\xi)}\right)\right] = c(\varepsilon)^2\exp\left(\varepsilon\left|\int_x^t\frac{d\xi}{d(\xi)}\right|\right);
 \end{align*}
 \begin{align*}
  \frac{d(t)}{d(x)}&=\frac{d(x_0)}{d(x)}\cdot\frac{d(t)}{d(x_0)}\ge \frac{m(\varepsilon)}{M(\varepsilon)} \exp\left(-\varepsilon\left|\int_x^t\frac{d\xi}{d(\xi)}\right|\right) \\
  &\ge c(\varepsilon)^{-2}\exp\left[-\varepsilon\left(\int_x^{x_0}\frac{d\xi}{d(\xi)}+\int_{x_0}^t\frac{d\xi}{d(\xi)}\right)\right]=c
  (\varepsilon)^{-2}\exp\left(-\varepsilon\left|\int_x^t\frac{d\xi}{d(\xi)}\right|\right).
 \end{align*}

 \end{proof}

\begin{proof}[Proof of \thmref{thm4.7} for $p\in(1,\infty)$] \textit{Necessity}.

 We need some auxiliary assertions.

 \begin{lem}\label{lem5.3}
 Let $p\in[1,\infty)$, $p'=p(p-1)^{-1}.$ Denote
 \begin{align}
 M_p(x)&=\left(\int_{-\infty}^x(\mu(t)v(t))^pdt\right)^{1/p}\cdot\left(\int_x^\infty\left(\frac{u(t)}{\theta(t)}\right)^{1/p'}dt
 \right),\quad
 x\in\mathbb R,\label{5.34}\\
 \tilde M_p(x)&=\left(\int_{-\infty}^x\left(\frac{v(t)}{\theta(t)}\right)^{p'}dt\right)^{1/p'}\cdot\left( \int_x^\infty(\mu(t)u(t))
 ^pdt\right)^{1/p},\quad x\in\mathbb R.\label{5.35}
 \end{align}
 Then we have the equalities (see \eqref{2.8}):
 \begin{align}
 M_p(x)&=\left[\int_{-\infty}^x\left(\sqrt{\rho(t)}\mu(t)\right)^p\exp\left(-\frac{p}{2}\int_t^x\frac{d\xi}{\rho(\xi)}\right) dt\right]^{1/p}\nonumber\\
 &\quad\cdot\left[\int_x^\infty\left(\frac{\sqrt{\rho(t)}}{\theta(t)}\right)^{p'}\exp\left(-\frac{p'}{2}\int_x^t\frac{d\xi}{\rho(\xi)}\right) dt\right]^{1/p'},\quad x\in\mathbb R,\label{5.36}
 \end{align}
  \begin{align}
 \tilde M_p(x)&=\left[\int_{-\infty}^x\left(\frac{\sqrt{\rho(t)}}{\theta(t)}\right)^{p'}  \exp\left(-\frac{p'}{2}\int_t^x\frac{d\xi}{\rho(\xi)}\right) dt\right]^{1/p'}
 \nonumber\\
 &\quad\cdot\left[\int_x^\infty\left(\mu(t)\sqrt{\rho(t)} \right)^{p}\exp\left(-\frac{p}{2}\int_x^t\frac{d\xi}{\rho(\xi)}\right) dt\right]^{1/p},\quad x\in\mathbb R.\label{5.37}
 \end{align}
 \end{lem}

\renewcommand{\qedsymbol}{\openbox}
 \begin{proof}
 Equalities \eqref{5.36} and \eqref{5.37} are proved in the same way. Consider, say, \eqref{5.36}. This equality can be obtained by substituting formulas \eqref{2.10} in \eqref{5.34}:
  \begin{align*}
 M_p(x)&=\left[\int_{-\infty}^x\left(\mu(t)\sqrt{\rho(t)}\right)^p\exp\left(\frac{p}{2}\int_{x_0}^t\frac{d\xi}{\rho(\xi)}
 \right) dt\right]^{1/p}\\
 &\quad\cdot\left[\int_x^\infty\left(\frac{\sqrt{\rho(t)}}{\theta(t)}\right) ^{p'} \exp\left(-\frac{p'}{2}\int_{x_0}^t\frac{d\xi}{\rho(\xi)}\right) dt\right]^{1/p'} \\
 &=\left[\int_{-\infty}^x\left(\mu(t)\sqrt{\rho(t)}\right)^p\exp\left(-\frac{p}{2}\int^{x}_t\frac{d\xi}{\rho(\xi)}
 \right) \cdot\exp\left(\frac{p}{2}\int_{x_0}^x\frac{d\xi}{\rho(\xi)}\right)dt\right]^{1/p}\\
  &\quad\cdot\left[\int_x^\infty\left(\frac{\sqrt{\rho(t)}}{\theta(t)}\right) ^{p'} \exp\left(-\frac{p'}{2}\int_x^t\frac{d\xi}{\rho(\xi)}\right)\cdot\exp\left(-\frac{p'}{2}\int_{x_0}^x\frac{d\xi}{\rho(\xi)}
  \right) dt\right]^{1/p'}\\
  &=\left[\int_{-\infty}^x\left(\mu(t)\sqrt{\rho(t)}\right)^p\exp\left(-\frac{p}{2}\int_{t}^x\frac{d\xi}{\rho(\xi)}
 \right) dt\right]^{1/p}\\
 &\quad\cdot\left[
 \int_x^\infty \left(\frac{\sqrt{\rho(t)}}{\theta(t)}\right) ^{p'}\exp\left(-\frac{p'}{2}
 \int_{x}^t\frac{d\xi}{\rho(\xi)} \right)dt  \right]^{1/p'}.
 \end{align*}
\end{proof}

Let us introduce some more notation:
\begin{equation}\label{5.38}
\varphi(x,t)=\begin{cases}
\displaystyle\frac{\mu(x)v(x)}{\mu(t)v(t)},&\quad\text{if}\quad x\le t\\ \\
\displaystyle\frac{\mu(t)v(t)}{\mu(x)v(x)},&\quad\text{if}\quad x\ge t\end{cases},\qquad
\psi(x,t)=\begin{cases}
\displaystyle\frac{\theta(x)u(t)}{\theta(t)u(x)},&\quad\text{if}\quad x\le t\\ \\
\displaystyle\frac{\theta(t)u(x)}{\theta(x)u(t)},&\quad\text{if}\quad x\ge t\end{cases}.
 \end{equation}

\begin{lem}\label{lem5.4}
Under the hypotheses of the theorem, for a given $\varepsilon>0$ and for all $t,x\in\mathbb R,$ we have the inequality
\begin{equation}\label{5.39}
\max\{\varphi(x,t);\psi(x,t)\}\le c(\varepsilon)\exp\left(\left(\sqrt 2\varepsilon-\frac{1}{2}\right)\left|\int_x^t\frac{d\xi}{\rho(\xi)}\right|\right).
\end{equation}
\end{lem}

\begin{proof} We will check inequality \eqref{5.39} for the function $\varphi$ (for the function $\psi$ the proof of \eqref{5.39} is similar). Below we use
\eqref{2.10}, \eqref{2.12} and \eqref{4.10}. Let $x\ge t.$ Then
\begin{align*}
\frac{\mu(t)}{\mu(x)}\cdot\frac{v(t)}{v(x)}&=\frac{\mu(t)}{\mu(x)}\cdot\sqrt{\frac{\rho(t)}{\rho(x)}}\exp\left(-\frac{1}{2}\int_t^x\frac{d\xi}{\rho(\xi)}
\right)\le c\frac{\mu(t)}{\mu(x)}\cdot\sqrt{\frac{d(t)}{d(x)}}\exp\left(-\frac{1}{2}
\int_t^x\frac{d\xi}{\rho(\xi)}\right)\\
&\le c(\varepsilon)\exp\left(\varepsilon\int_t^x\frac{d\xi}{d(\xi)}-\frac{1}{2}\int_t^x\frac{d\xi}{\rho(\xi)}\right)\le
c(\varepsilon)\exp\left(\left(\sqrt 2\varepsilon-\frac{1}{2}\right)\int_t^x\frac{d\xi}{\rho(\xi)}\right)\\
&=c(\varepsilon)\exp\left(\left(\sqrt 2\varepsilon-\frac{1}{2}\right)\left|\int_x^t\frac{d\xi}{\rho(\xi)}\right|\right);
\end{align*}

 Similarly, for $x\le t,$ we have:
    \begin{align*}
    \frac{\mu(x)}{\mu(t)}\cdot\frac{v(x)}{v(t)}&=\frac{\mu(x)}{\mu(t)}\sqrt{\frac{\rho(x)}{\rho(t)}}\exp\left(-\frac{1}{2}
    \int_x^t\frac{d\xi}{\rho(\xi)}\right)\le c\frac{\mu(x)}{\mu(t)}\sqrt{\frac{d(x)}{d(t)}}\exp\left(-\frac{1}{2}\int_x^t\frac{d\xi}{\rho(\xi)}\right)\\
   &\le c(\varepsilon)\exp\left(\varepsilon\int_t^x\frac{d\xi}{d(\xi)}-\frac{1}{2}
   \int_t^x\frac{d\xi}{\rho(\xi)}\right)\le
   c(\varepsilon)\exp\left(\left(\sqrt 2\varepsilon-\frac{1}{2}\right) \int_t^x\frac{d\xi}{\rho(\xi)} \right)\\
   &=c(\varepsilon)\exp\left(\left(\sqrt 2\varepsilon-\frac{1}{2}\right)\left|\int_x^t\frac{d\xi}{\rho(\xi)}\right|\right).
    \end{align*}
\end{proof}

\begin{lem}\label{lem5.5}
Under conditions \eqref{1.1} and \eqref{2.1}, we have the inequality
\begin{equation}\label{5.40}
\int_{x-d(x)}^{x+d(x)}\frac{d\xi}{d(\xi)}\le 8,\quad\forall\ x\in\mathbb R.
\end{equation}
\end{lem}

\begin{proof}
Estimate \eqref{5.40} follows from \eqref{3.2}:
$$\int_{x-d(x)}^{x+d(x)}\frac{d\xi}{d(\xi)}=\int_{x-d(x)}^{x+d(x)}\frac{d(x)}{d(\xi)}\cdot\frac{d\xi}{d(x)}\le\int_{x-d(x)} ^{x+d(x)}4\frac{d\xi}{d(x)}=8.
$$
\end{proof}

\begin{lem}\label{lem5.6}
Under the hypotheses of the theorem,   we have the inequalities
\begin{equation}\label{5.41}
 c^{-1}\le\frac{\mu(t)}{\mu(x)},\quad \frac{\theta(t)}{\theta(x)}\le c;\quad\text{if}\quad t\in[x-d(x),x+d(x)],\quad x\in\mathbb R.
\end{equation}
\end{lem}

\begin{proof}
We will only check inequalities \eqref{5.41} for the function $\mu$ (the proof of \eqref{5.41} for the function $\theta$ is similar). In \eqref{4.10}, set $\varepsilon=\frac{1}{2}.$ Now for $|t-x|\le d(x),$ $x\in\mathbb R$, we use \eqref{3.2}, \eqref{4.10} and \eqref{5.40}:
\begin{align*}
\frac{\mu(t)}{\mu(x)}&\le c\sqrt{\frac{d(x)}{d(t)}}\exp\left(\frac{1}{2}\left|\int_x^t\frac{d\xi}{d(\xi)}\right|\right) \le c\exp\left(\frac{1}{2}\int_{x-d(x)}^{x+d(x)}\frac{d\xi}{d(\xi)}\right)\le c<\infty,\\
 \frac{\mu(t)}{\mu(x)}&\ge c^{-1}\sqrt{\frac{d(x)}{d(t)}}\exp\left(-\frac{1}{2}\left|\int_x^t\frac{d\xi}{d(\xi)}\right|\right) \ge c^{-1}\exp\left(-\frac{1}{2}\int_{x-d(x)}^{x+d(x)}\frac{d\xi}{d(\xi)}\right)\le c^{-1}>0.
\end{align*}
\end{proof}

Let us now go over to the theorem. Since condition \eqref{2.1} holds, by \thmref{thm2.2}, a FSS $\{u,v\}$ of equation \eqref{2.4} is defined, and thus the operator $S$ (see \eqref{4.3}) is also defined. Since the pair $\{L_{p,\mu};L_{p,\theta}\}$ is admissible for \eqref{1.1}, by \thmref{thm4.3} the operator $S:L_p\to L_p,$ $p\in[1,\infty)$ is bounded. Then so are the operators $S_i: L_p\to L_p$, $i=1,2$ (see \eqref{5.25}). Let $p\in(1,\infty).$ Consider, say, the operator $S_2: L_p\to L_p.$ Since it is bounded, we have $M_p<\infty$ by \thmref{thm2.7} (see \eqref{5.27} and \eqref{5.34}. Below we use this fact together with \lemref{lem2.1}, \eqref{2.12}, \eqref{5.40} and \eqref{5.41}:
\begin{align*}
\infty&>M_p=\sup_{x\in\mathbb R} M_p(x)=\sup_{x\in\mathbb R}\left(\int_{-\infty}^x(\mu(t)v(t))^pdt\right)^{1/p}\left( \int_x^\infty\left(\frac{u(t)}{\theta(t)}\right)^{p'}dt\right)^{1/p'}\\
&=\sup_{x\in\mathbb R}\left[\int_{-\infty}^x\left(\sqrt{\rho(t)\mu(t)}\right)^p\exp\left(-\frac{p}{2}\int_t^x\frac{d\xi} {\rho(\xi)}\right)dt\right]^{1/p}\\
&\quad  \cdot\left[\int_x^\infty\left(\frac{\sqrt{\rho(t)}}{\theta(t)}\right)^{p'}
  \exp\left(- \frac{p'}{2}\int_x^t\frac{d\xi}{\rho(\xi)}
  \right)
  dt\right]^{1/p'}\\
   &\ge\sup_{x\in\mathbb R}\left[\int_{x-d(x)}^x\left(\sqrt{\rho(t)}\mu(t)\right)^p\exp\left(-\frac{p}{2}\int_t^x\frac{d\xi} {\rho(\xi)}\right)dt\right]^{1/p}\\
 &\quad \cdot\left[\int_x^{x+d(x)}\left(\frac{\sqrt{\rho(t)}}{\theta(t)}\right)^{p'}\exp\left(-\frac{p'}{2} \int_x^t\frac{d\xi}{\rho(\xi)}\right)dt\right]^{1/p'}\\
 &\ge c^{-1}\sup_{x\in\mathbb R}\left[\int_{x-d(x)}^x\left(\sqrt{d(t)}\mu(t)\right)^p\exp\left(-\sqrt 2p\int_t^x\frac{d\xi}{d(\xi)}\right)dt\right]^{1/p}\\
 &\quad \cdot\left[\int_x^{x+d(x)}\left(\frac{\sqrt{d(t)}}{\theta(t)}\right)^{p'} \exp\left(-\sqrt 2p'\int_x^t\frac{d\xi}{d(\xi)}\right)dt\right]^{1/p'}\\
 &\ge c^{-1}\sup_{x\in\mathbb R}\left[\int_{x-d(x)}^x\left(\sqrt{d(t)}\mu(t)\right)^p\exp\left(-\sqrt{2p}\int_{x-d(x)}^{x+d(x)}\frac{d\xi}{d(\xi)}\right) dt\right]^{1/p}\\
 &\quad \cdot\left[\int_x^{x+d(x)}\left(\frac{\sqrt{d(t)}}{\theta(t)}\right)^{p'}\exp\left(-\sqrt2p'\int_{x-d(x)}^{x+d(x)} \frac{d\xi}{d(\xi)}\right)dt\right]^{1/p'}\\
 &\ge c^{-1}\sup_{x\in\mathbb R}\frac{\mu(x)}{\theta(x)}d^2(x)=c^{_1} m(q,\mu,\theta),
\end{align*}
as required. Let $p=1.$ Since the operator $S:L_1\to L_1$ is bounded (\thmref{thm4.3}), so are the operators $S_i:L_1\to L_1, $ $i=1,2$ (see \lemref{lem5.2}). Let, say, $i=2.$ Below we use \thmref{thm2.9}, \eqref{5.22}, \lemref{lem2.1}, \eqref{2.10}, \eqref{2.12}, \eqref{5.40}, \eqref{5.41} and \eqref{3.2}:
\begin{align*}
\infty&>\|S_2\|_{1\to1}=\sup_{x\in\mathbb R}\frac{u(x)}{\theta(x)}\int_{-\infty}^x\mu(t)v(t)dt\ge\sup_{x\in\mathbb R}\frac{u(x)}{\theta(x)}\int_{x-d(x)}^x\mu(t)v(t)dt\\
&=\sup_{x\in\mathbb R}\frac{\sqrt{\rho(x)}}{\theta(x)}\int_{x-d(x)}^x\mu(t)]\sqrt{\rho(t)}\exp\left(-\frac{1}{2}\int_t^x\frac{d\xi}{\rho(\xi)}\right)\\
&\ge c^{-1}\sup_{x\in\mathbb R}\frac{\sqrt{d(x)}}{\theta(x)}\int_{x-d(x)}^x\mu(t)\sqrt{d(t)}\exp\left(-\sqrt2\int_t^x\frac{d\xi}{d(\xi)}\right)dt\\
&\ge c^{-1}\sup_{x\in\mathbb R}\frac{\sqrt{d(x)}}{\theta(x)}\int_{x-d(x)}^x\mu(t)\sqrt{d(t)}\exp\left(-\sqrt2\int_{x-d(x)}^{x+d(x)}\frac{d\xi}{d(\xi)}\right)dt\\
&\ge c^{-1}\sup_{x\in\mathbb R}\frac{\sqrt{d(x)}}{\theta(x)}\int_{x-d(x)}^x\mu(t)\sqrt{d(t)}dt\ge c^{-1}\sup_{x\in\mathbb R}\frac{\mu(x)}{\theta(x)}d^2(x)=c^{-1}m(q,\mu,\theta).
\end{align*}
\renewcommand{\qedsymbol}{}

\begin{proof}[Proof of \thmref{thm4.7}] \textit{Sufficiency.}

It is enough to show that the operators $S_i: L_p\to L_p,$ $p\in[1,\infty),$ $i=1,2,$ are bounded. Indeed, then so is the operator $S:L_p\to L_p,$ $p\in[1,\infty)$ (see \eqref{5.25}), and then by \thmref{thm4.3} the pair $\{L_{p,\mu};\, L_{p,\theta}$ is admissible for \eqref{1.1}. Both operators $S_i,$ $i=1,2,$ are treated in the same way, and therefore below we only consider the operator $S_2$ (see \eqref{4.5}, \eqref{5.22}). Below, when estimating $\|S_2\|_{p\to p},$ $p\in(1,\infty)$, we use \thmref{thm2.7}, \eqref{5.22}, \eqref{5.34}, \eqref{5.38}, \eqref{5.39}, \eqref{4.10} for $\varepsilon=1/4\sqrt2,$ \eqref{2.12} and \eqref{4.11}:

\begin{align*}
&\|S_2\|_{p\to p} \le c(p)\sup_{x\in\mathbb R}\left[\int_{-\infty}^x(\mu(t)v(t))^pdt\right]^{1/p}\cdot\left[\int_x^\infty\left(\frac{u(t)}{\theta(t)}\right)^{p'}dt\right] ^{1/p'}\\
&=c(\varepsilon)\sup_{x\in\mathbb R}(\mu(x)v(x))^{1/p'}\left[\int_{-\infty}^x\left(\frac{\mu(t)v(t)}{\mu(x)v(x)}\right)^{p-1}(\mu(t)v(t))dt\right]^{1/p}\\
&\quad\cdot\left(\frac{u(x)}{\theta(x)}\right)^{1/p}\left[\int_x^\infty\left(\frac{u(t)}{\theta(t)}\cdot \frac{\theta(x)}{u(x)}\right)^{p'-1}\left(\frac{u(t)}{\theta(t)}\right)dt\right]^{1/p'}\\
&=c(\varepsilon)\sup_{x\in\mathbb R}\left[\frac{u(x)}{\theta(x)}\int_{-\infty}^x\varphi(x,t)^{p-1}(\mu(t)v(t))dt\right]^{1/p}\\
&\quad \cdot\left[\mu(x)v(x)\int_x^\infty\psi(x,t)^{p'-1}\left(\frac{u(t)}{\theta(t)}\right) dt\right]^{1/p'}\\
&\le c(\varepsilon)\sup_{x\in\mathbb R}\left[\int_{-\infty}^x\left(\frac{u(x)}{\theta(x)}\cdot\frac{\theta(t)}{u(t)}\right)\cdot\varphi(x,t)^{p-1}\frac{ \mu(t)}{\theta(t)}\rho(t)dt\right]^{1/p}\\
&\quad \cdot\left[\int_x^\infty\left(\frac{\mu(x)v(x)}{\mu(t)v(t)}\right)\cdot\psi(x,t)^{p'-1}\cdot\frac{\mu(t)}{\theta(t)}\rho(t)dt
\right]^{1/p'}\\
&=c(\varepsilon)\sup_{x\in\mathbb R}\left[\int_{-\infty}^x\psi(x,t)\cdot\varphi(x,t)^{p-1}\frac{\mu(t)}{\theta(t)}\rho(t)dt\right]^{1/p}\\
&\quad\cdot \left[\int_x^\infty\varphi(x,t)\psi(x,t)^{p'-1}\frac{\mu(t)}{\theta(t)}\rho(t)dt\right]^{1/p'}\\
&\le c(\varepsilon)\sup_{x\in\mathbb R}
\left[
\int_{-\infty}^x\left(\frac{\mu(t)}{\theta(t)}d^2(t)\right)\cdot\left(\frac{\rho(t)}{d(t)}
\right)^2
  \cdot\frac{1}{\rho(t)}
  \exp\left(\left(\sqrt 2\varepsilon-\frac{1}{2}\right)
  p\int_t^x\frac{d\xi}{\rho(\xi)}\right)dt
 \right]^{1/p}\\
&\quad\cdot\left[\int_x^\infty\left(\frac{\mu(t)}{\theta(t)}d^2(t)\right)\cdot\left(\frac{\rho(t)}{d(t)}\right)^2
\frac{1}{\rho(t)}
\exp\left(\left(\sqrt2\varepsilon-\frac{1}{2}\right){p'}\int_x^t\frac{d\xi}{\rho(\xi)}
\right)dt\right]^{1/p'}\\
&\le c(\varepsilon)m(q,\mu,\theta)\sup_{x\in\mathbb R}\left[\int_{-\infty}^x\frac{1}{\rho(t)}\exp\left(-\frac{p}{4}\int_t^x
\frac{d\xi}{\rho(\xi)}\right)dt\right]^{1/p}\\
&\quad\cdot\left[\int_x^\infty\frac{1}{\rho(t)}\exp\left(-\frac{p'}{4}\int_x^t\frac{d\xi}{\rho(\xi)}\right)dt\right]^{1/p'} \le cm(q,\mu,\theta)<\infty.
\end{align*}

Consider the case $p=1.$ Below, when estimating $\|S\|_{1\to1}, $ we use \eqref{2.21}, \eqref{4.3}, \eqref{2.11} and \eqref{4.10} for $\varepsilon=1/4\sqrt2$, and \eqref{2.12}:
\begin{align*}
 \|S\|_{1\to1}&=\sup_{x\in\mathbb R}\frac{1}{\theta(x)}\int_{-\infty}^\infty\mu(t)G(x,t)dt=\sup_{x\in\mathbb R}\frac{\sqrt {\rho(x)}}{\theta(x)}\int_{-\infty}^\infty\mu(t)\sqrt{\rho(t)}\exp\left(-\frac{1}{2}\left|\int_x^t\frac{d\xi}{\rho(\xi)}\right|
 \right)dt\\
 &=\sup_{x\in\mathbb R}\int_{-\infty}^\infty\left(\frac{\mu(t)}{\theta(t)}d^2(t)\right)\cdot\left(\frac{\rho(t)}
 {d(t)}\right)^2
 \frac{\theta(t)\sqrt{d(x)}}
{\theta(x)\sqrt{d(t)}}\cdot
\sqrt{\frac{ \rho(x)}{d(x) } \cdot\frac{d(t)}{\rho(t)}}\cdot\frac{1}{\rho(t)}
 \exp\left(-\frac{1}{2}\left|\int_x^t\frac{d\xi}{\rho(\xi)}\right| \right)dt\\
&\le cm(q,\mu,\theta)\int_{-\infty}^\infty\frac{1}{\rho(t)}\exp\left(\varepsilon
\left|\int_x^t\frac{d\xi}{d(\xi)}\right|-\frac{1}{2}\left|
\int_x^t\frac{d\xi}{\rho(\xi)}\right|\right)dt\\
&\le cm(q,\mu,\theta)\int_{-\infty}^\infty\frac{1}{\rho(t)}\exp\left(\left(\sqrt 2\varepsilon-\frac{1}{2}\right)\left|\int_x^t\frac{d\xi}{\rho(\xi)}\right|\right)dt\\
&=cm(q,\mu,\theta)\int_{-\infty}^\infty\frac{1}{\rho(t)} \exp\left(-\frac{1}{4}\left|\int_x^t\frac{d\xi}{\rho(\xi)}\right|\right)dt=
cm(q,r,\mu)<\infty.
\end{align*}
Thus the operator $S:L_p\to L_p,$ $p\in[1,\infty)$ is bounded, and it remains to refer to \thmref{thm4.3}.

 \end{proof}
 \end{proof}

\renewcommand{\qedsymbol}{\openbox}

\begin{proof}[Proof of \lemref{lem4.8}]

Fix $\varepsilon>0$ and choose $x_0=x_0(\varepsilon)>>1$ in order to have the inequalities
\begin{equation}\label{5.42}
-\frac{\varepsilon}{3}\le\frac{\mu'(\xi)}{\mu(\xi)}d(\xi);\quad d'(\xi)\le\frac{\varepsilon}{3}\quad\text{for all}\quad |\xi|\ge x_0.
\end{equation}
{}From \eqref{5.42}, one can easily deduce the estimates
\begin{equation}\label{5.43}
-\frac{2}{3}\le\frac{(\mu(\xi)d(\xi))'}{\mu(\xi)d(\xi)}\le \frac{2\varepsilon}{3}\cdot \frac{1}{d(\xi)}\quad\text{for all}\quad |\xi|\ge x_0.
\end{equation}
Let, say, $t\ge x\ge x_0.$ Then from \eqref{5.43}, we obtain
\begin{equation}\label{5.44}
\exp\left(-\frac{2\varepsilon}{3}\int_x^t\frac{d\xi}{d(\xi)}\right)\le\frac{\mu(t)d(t)}{\mu(x)d(x)}\le\exp\left(\frac{2\varepsilon}
{3}\int_x^t\frac{d\xi}{d(\xi)}\right),\quad t\ge x\ge x_0.
\end{equation}

Let us write \eqref{5.44} in a different way:
$$\sqrt{\frac{d(x)}{d(t)}}\exp\left(-\frac{2\varepsilon}{\varepsilon}\int_x^t\frac{d\xi}{d(\xi)}\right)\le\frac{\mu(t)}{\mu(x)}
\sqrt{\frac{d(t)}{d(x)}}\le\sqrt{\frac{d(x)}{d(t)}}
\exp\left(\frac{2\varepsilon}{3}\int_x^t\frac{d\xi}{d(\xi)}\right),\quad t\ge x\ge x_0.$$
We now combine the latter estimates with inequalities \eqref{4.8} written for $\frac{2\varepsilon}{3}$ instead of $\varepsilon:$
$$ c\left(\frac{2\varepsilon}{3}\right)^{-1/2}\exp\left(-\frac{2\varepsilon}{3}\int_x^t\frac{d\xi}{d(\xi)}\right)\le\sqrt{ \frac{d(x)}{d(t)}}
\le c\left(\frac{2\varepsilon}{3}\right)^{1/2}\exp\left(\frac{\varepsilon}{3}\int_x^t\frac{d\xi}{d(\xi)}\right).
$$
We easily obtain that for $t\ge x\ge x_0$ we have the inequalities
$$c\left(\frac{2}{3}\varepsilon\right)^{-1/2}\exp\left(-\varepsilon\int_x^t\frac{d\xi}{d(\xi)}\right)\le\frac{\mu(t)}{\mu(x)}
 \sqrt{\frac{d(t)}{d(x)}}\le c\left(\frac{2\varepsilon}{3}\right)^{1/2}\exp\left(\varepsilon\int_x^t\frac{d\xi}{d(\xi)}\right),$$
as required. The cases $x\ge t\ge x_0$ and the cases $t\le x\le-x_0,$ $x\le t\le-x_0$ are considered in a similar way.
We then continue the proof as in \lemref{lem4.5}, with obvious modifications, similar to those presented above.
\end{proof}

\section{Example}

In this final section, we consider equation \eqref{1.1} with
\begin{equation}\label{6.1}
q(x)=\frac{1}{\sqrt{1+x^2}}+\frac{\cos(e^{|x|})}{\sqrt{1+x^2}},\qquad x\in\mathbb R.
\end{equation}
Using the results obtained above, we show that the following assertions hold:

\begin{enumerate}
\item[A)] Equation \eqref{1.1} in the case of \eqref{6.1} is not correctly solvable in $L_p,$ for any $p\in[1,\infty);$
\item[B)] For equation \eqref{1.1} in the case of \eqref{6.1}, for any $p\in[1,\infty)$, the following pair of spaces $\{L_{p,\mu};L_{p,\theta}\}$ is admissible, where
  \begin{equation}\label{6.2}
\mu(x)=\frac{1}{\sqrt{1+x^2}\ln(2+x^2)},\qquad \theta(x)=\frac{1}{ \ln(2+x^2)},\qquad x\in\mathbb R.
\end{equation}
\end{enumerate}

\begin{remark*}
Below we present an algorithm for the study of \eqref{1.1} for a given pair of spaces (cases \eqref{6.1} and $\{L_p,L_p\}$ and $\{L_{p,\mu};L_{p,\theta}\}$ in the case of \eqref{6.2}). We do not consider the question of the description of all pairs of spaces admissible for \eqref{1.1} in the case of \eqref{6.1}.
\end{remark*}

For the reader's convenience, we enumerate the main steps of the proof of assertions A) and B). Note that since the functions in \eqref{6.1} and \eqref{6.2} are even, all proofs are only given for $x\in[0,\infty)$ or for  $x\in[x_0,\infty),$ $x_0\gg1$.

\noindent 1) \textit{Checking condition \eqref{2.1}}.

Let us check that in the case of \eqref{6.1} condition \eqref{2.1} holds. Assume to the contrary that there is $x_0\in\mathbb R $ such that
 \begin{equation}\label{6.3}
\int_{x_0}^\infty q(t)dt=0.
\end{equation}
The function $q$ in \eqref{6.1} is continuous and non-negative. Therefore, from \eqref{6.3} it follows that $q(t)\equiv0$ for $t\in[x_0,\infty)$ which is obviously false. This contradiction implies \eqref{2.1}.

\noindent 2) \textit{Existence of the function $d(x), $ $x\in\mathbb R$, and its estimates}.

From 1) and \lemref{lem2.1}, it follows that the function $d(x)$ is defined for all $x\in\mathbb R.$ To obtain its estimates, we use \thmref{thm3.7}. Denote (see \eqref{3.7} and \eqref{3.8})
\begin{gather}
q_1(x)=\frac{1}{\sqrt{1+x^2}};\qquad q_2(x)=\frac{\cos(e^{|x|})}{\sqrt{1+x^2}},\qquad x\in\mathbb R;\label{6.4}\\
A(x)=\left[0,2\sqrt[4]{1+x^2}\right];\qquad \omega(x)=\left[x-2\sqrt[4]{1+x^2},x+2\sqrt[4]{1+x^2}\right],\qquad x\in\mathbb R.\label{6.5}
\end{gather}

Let us check \eqref{3.11} for the function $\varkappa_1$ (see \eqref{3.9}):
\begin{align}
\varkappa_1(x)&=\frac{1}{q_1(x)^{3/2}}\sup_{t\in A(x)}\left|\int_{x-t}^{x+t}q_1''(\xi)d\xi\right|
=(1+x^2)^{3/4}\sup_{t\in A(x)}\left|\int_{x-t}^{x+t}\left(\frac{1}{\sqrt{1+\xi^2}}\right)''d\xi\right|\nonumber\\
&=(1+x^2)^{3/4}\sup_{t\in A(x)}\left|\int_{x-t}^{x+t}\frac{1-2\xi^2}{1+\xi^2}\cdot \frac{d\xi}{(1+\xi^2)^{3/2}}\right|.\label{6.6}
\end{align}
Note the obvious inequalities
 \begin{equation}\label{6.7}
\left|\frac{1-2\xi^2}{1+\xi^2}\right|\le\frac{1+2\xi^2}{1+\xi^2}\le 2,\qquad \xi\in\mathbb R.
\end{equation}

In addition, for $\xi\in A(x),$ $x\gg1,$ we have
\begin{align}
\frac{1+\xi^2}{1+x^2}&\le 1+\frac{|\xi-x|\, |\xi+x|}{1+x^2}\le 1+c\frac{x^{3/2}}{1+x^2}\le 2;\label{6.8}\\
\frac{1+\xi^2}{1+x^2}&\ge 1-\frac{|\xi-x|\, |\xi+x|}{1+x^2}\ge 1 -c\frac{x^{3/2}}{1+x^2}\ge \frac{1}{2}.\label{6.9}
\end{align}

{}From \eqref{6.6}, \eqref{6.7}, \eqref{6.8} and \eqref{6.9}, it now follows that
\begin{align*}
\varkappa_1(x)&\le (1+x^2)^{3/4}\sup_{t\in A(x)}\left[\int_{x-t}^{x+t}\left|\frac{1-2\xi^2}{1+\xi^2}\right|\frac{1}{(1+x^2)^{3/2}}\left(\frac{1+x^2}{1+\xi^2}\right)^{3/2}d\xi\right]\\
&\le c\frac{(1+x^2)^{3/4}}{(1+x^2)^{3/2}}\sup_{t\in A(x)}\left|\int_{x-t}^{x+t}1d\xi\right|=\frac{c}{\sqrt{1+x^2}}\to 0,\quad x\to\infty.
\end{align*}

Let us now check \eqref{3.11} for $\varkappa_2(x)$, $x\gg1.$ First show that for $x\gg1$ we have the inequality (see \eqref{6.5}):
 \begin{equation}\label{6.10}
 \sup_{[\alpha,\beta]\subseteq\omega(x)}\left|\int_\alpha^\beta\frac{\cos e^t}{\sqrt{1+t^2}}\right|\le c\frac{e^{-x/2}}{\sqrt{1+x^2}},\qquad x\gg 1.
 \end{equation}

We need the following simple assertions, given without proof:
\begin{enumerate}
\item [a)] $x-\sqrt[4]{1+x^2}\to\infty$ as $x\to\infty;$
\item[b)] the function $\varphi(\xi)$ where
$$\varphi(\xi)=\frac{e^{-\xi}}{\sqrt{1+\xi^2}},\qquad\xi\in\mathbb R$$
is monotone decreasing for all $\xi\in\mathbb R.$
\end{enumerate}

Let $t$ be any point in the interval $(\alpha,\beta).$ Below we use assertions a), b) and the second mean theorem   (see  \cite{15}):
\begin{align}
 \sup_{[\alpha,\beta]\subseteq \omega(x)}&\left|\int_\alpha^\beta\frac{\cos e^\xi}{\sqrt{1+\xi^2}}d\xi\right|
 =\sup_{[\alpha,\beta]\subseteq\omega(x)}\left|\int_\alpha^\beta\frac{e^{-\xi}}{\sqrt{1+\xi^2}}
 (e^\xi cos e^\xi)d\xi\right|\nonumber\\
 &=\sup_{[\alpha,\beta]\subseteq\omega(x)}\frac{e^{-\alpha}}{\sqrt{1+\alpha^2}}\left|\int_\alpha^t e^\xi\cos e^\xi d\xi\right|\le c\left.\frac{c^{-\xi}}{\sqrt{1+\xi^2}}\right|_{\xi=x-2\sqrt[4]{1+x^2}}\le c\frac{e^{-x/2}}{\sqrt{1+x^2}}.\label{6.12}
\end{align}

Now, from \eqref{6.12} for $x\gg1$ we obtain
\begin{align*}
\varkappa_2(x)&=\frac{1}{\sqrt{q_1(x)}}\sup_{t\in A(x)}\left|\int_{x-t}^{x+t}q_2(\xi)d\xi\right|=\sqrt[4]{1+x^2}\sup_{t\in A(x)}\left|\int_{x-t}^{x+t}\frac{\cos e^\xi}{\sqrt{1+\xi^2}}d\xi\right|\\
&\le \sqrt[4]{1+x^2}\sup_
{[\alpha,\beta]\subseteq \omega(x)}\left|\int_\alpha^\beta\frac{\cos e^\xi}{\sqrt{1+\xi^2}} d\xi\right|\le c
\frac{e^{-x/2}}{\sqrt{1+x^2}}\quad\Rightarrow \quad\eqref{3.11}.
\end{align*}

Since \eqref{3.11} is proven, by \thmref{thm3.7} we obtain
\begin{gather}
d(x)=\sqrt[4]{1+x^2}(1+\varepsilon(x)),\qquad |\varepsilon(x)|\le 2(\varkappa_1(x)+\varkappa_2(x)),\quad |x|\gg 1,\label{6.13}\\
c^{-1}\sqrt[4]{1+x^2}\le d(x)\le c \sqrt[4]{1+x^2},\qquad x\in\mathbb R.\label{6.14}
\end{gather}

\noindent 3) \textit{Proof of assertion A)}.

{}From \eqref{6.14}, it follows that $d_0=\infty$ (see \eqref{2.3} and \eqref{2.14}). It remains to refer to \thmref{thm2.6}. \hfil\quad $\square$

Let us now go to assertion B).

\noindent 4) \textit{Checking the inclusion $q\in H.$}

To prove \eqref{4.6}, we need estimates of $\tau_1(x)$ and $\tau_2(x)$ for $x\gg 1$ where (see \eqref{6.1} and \eqref{6.4})
\begin{align}
\tau_1(x)&=\left|\int_0^{\sqrt 2d(x)}(q_1(x+t)-q_1(x-t))dt\right|;\label{6.15}\\
\tau_2(x)&=\left|\int_0^{\sqrt 2d(x)}(q_2(x+t)-q_2(x-t))dt\right|.\label{6.16}
\end{align}

To estimate $\tau_1(x)$, we use below \eqref{6.7}, \eqref{6.8}, \eqref{6.9} and \eqref{6.13}:
\begin{align}
\tau_1(x)&=\left|\int_0^{\sqrt 2d(x)}\left(\int_{x-t}^{x+t}q_1'(\xi)d\xi\right)dt\right|\le\sqrt 2d(x)\sup_{|t|\le\sqrt 2d(x)} \left|\int_{x-\xi}^{x+\xi}q_1'(t)dt\right|\nonumber\\
&\le c\sqrt[4]{1+x^2}\sup_{|\xi|\le 2\sqrt[t]{1+t^2}}\left|\int_{x-\xi}^{x+\xi}\frac{t}{\sqrt{1+t^2}}\cdot\frac{1+x^2}{1+t^2}\cdot\frac{dt}{1+x^2}\right|\nonumber\\
&\le c\frac{\sqrt[4]{1+x^2}}{1+x^2}\sup_{|t|\le 2\sqrt[4]{1+x^2}}|t|\le \frac{c}{\sqrt{1+x^2}},\qquad x\gg1.\label{6.17}
\end{align}

The estimate for $\tau_2(x),$ $x\gg1,$ follows from \eqref{6.10} and \eqref{6.1}:
\begin{align}
|\tau_2(x)|&\le \left|\int_0^{\sqrt 2d(x)}q_2(x+t)dt\right|+\left|\int_0^{\sqrt 2d(x)}q_2(x-t)dt\right|\nonumber\\
&=\left|\int_{x-\sqrt 2d(x)}^xq_2(\xi)d\xi\right|+\left|\int_x^{x+\sqrt 2d(x)}q_2(\xi)d\xi\right|\nonumber\\
&\le 2\sup_{[\alpha,\beta]\subseteq\omega(x)}\left|\int_\alpha^\beta\frac{\cos e^\xi}{1+\xi^2}d\xi\right|\le c\frac{e^{-x}}{\sqrt{1+x^2}},\quad x\gg 1.\label{6.18}
\end{align}
From \eqref{6.17}, \eqref{6.18} and \eqref{6.14}, we obtain \eqref{4.6}, and therefore $q\in H.$

\noindent 5. \textit{Checking that the weights $\mu(x)$ and $\theta(x)$ agree with the function $q.$}

Equalities \eqref{4.12} for the functions $\mu(x)$ and $1/\theta(x)$ (see \eqref{6.2}) are easily proved with the help of estimates \eqref{6.14}.

\noindent 6. \textit{Proof of assertion B).}

Below we use \thmref{thm4.7}. Let us check that in case \eqref{6.2} requirements \eqref{4.2} are satisfied.
Let $x_0\gg1.$ Then
\begin{align*}
\int_0^\infty \mu(t)dt&=\int_0^\infty\frac{dt}{\sqrt{1+t^2}\ln(2+t^2)}\ge\int_{x_0}^\infty\frac{1}{t\sqrt{1+t^{-2}}}\cdot \frac{dt}{2\ln t+\ln(1+2t^{-2})}\\
&\ge c^{-1}\int_{x_0}^\infty\frac{dt}{t\ln t}=\infty\quad\Rightarrow\quad \eqref{4.2}.
\end{align*}
Since the weights $\mu$ and $\theta$ agree with the function $q,$ and one has the relations (see assertion \eqref{6.14}):
$$m(q,\mu,\theta)=\sup_{x\in\mathbb R}\left(\frac{\mu(x)}{\theta(x)}d^2(x)\right)\le c\sup_{x\in\mathbb R}\left(\frac{\mu(x)}{\theta(x)}\sqrt{1+x^2}\right)=c<\infty,$$
assertion B) follows from \thmref{thm4.7}.

\hfill\quad\qed

\end{document}